\documentclass{amsart}

\usepackage{amssymb}
\usepackage{amsthm}
\usepackage{enumitem}

\theoremstyle{plain}\newtheorem{Theorem}{Theorem}[section]
\theoremstyle{plain}\newtheorem{Conjecture}[Theorem]{Conjecture}

\theoremstyle{plain}\newtheorem{Corollary}[Theorem]{Corollary}
\theoremstyle{plain}\newtheorem{Lemma}[Theorem]{Lemma}
\theoremstyle{plain}
\theoremstyle{definition}
\theoremstyle{definition}
\theoremstyle{definition}
\theoremstyle{definition}\newtheorem{Remark}[Theorem]{Remark}
\theoremstyle{definition}\newtheorem{Notation}[Theorem]{Notation}
\theoremstyle{definition}\newtheorem{Hypothesis}[Theorem]{Hypothesis}
\theoremstyle{definition}
\theoremstyle{definition}\newtheorem{Strategy}[Theorem]{Strategy}
\theoremstyle{definition}\newtheorem{Contents}[Theorem]{Contents}
\theoremstyle{definition}\newtheorem{Computations}[Theorem]{Computations}
\theoremstyle{definition}\newtheorem{Equivalences}[Theorem]{Equivalences}

\def\C{{\mathbb C}}

\def\Q{{\mathbb Q}}

\def\dim{\mathrm{dim}}           
           
\def\Ext{\mathrm{Ext}}           

\def\Hom{\mathrm{Hom}}

\def\Im{\mathrm{Im}}

\def\mod{\mathrm{mod}}
\def\rad{\mathrm{rad}}

\def\soc{\mathrm{soc}}

\newcommand{\SL}{\operatorname{SL}}

\newcommand{\GL}{\operatorname{GL}}

\begin{document}

Revised Version of 11 September 2011

\bigskip
\bigskip

\begin{center}
{\LARGE\bf  Brou{\'e}'s abelian defect group conjecture holds\\ 
\vspace*{0.3em} 
            for the sporadic simple Conway group {\sf Co}$_3$}
\end{center}

\bigskip

\begin{center}
   {\large
        {\bf Shigeo Koshitani}$^{\text a,*}$, 
        {\bf J{\"u}rgen M{\"u}ller}$^{\text b}$,
         {\bf Felix Noeske}$^{\text c}$  
   }
\end{center}

\bigskip

\begin{center}
  {\it
${^{\mathrm{a}}}$Department of Mathematics, Graduate School of Science, \\
Chiba University, Chiba, 263-8522, Japan \\
${^{\mathrm{b, \, c}}}$Lehrstuhl D f{\"u}r Mathematik,
RWTH Aachen University, 52062, Aachen, Germany}
\end{center}

\footnote
{
$^*$ Corresponding author. \\
\indent {\it E-mail addresses:} koshitan@math.s.chiba-u.ac.jp (S.Koshitani),
\\
juergen.mueller@math.rwth-aachen.de (J.M{\"u}ller),
Felix.Noeske@math.rwth-aachen.de (F.Noeske).
}

\hrule 

\smallskip
{\small \noindent{\bf Abstract}

\smallskip\noindent
In the representation theory of finite groups, there is a well-known and
important conjecture due to M.~Brou{\'e}. He conjectures that, for
any prime $p$, if a $p$-block $A$ of a finite group $G$ has an abelian
defect group $P$, then $A$ and its Brauer corresponding block $A_N$
of the normaliser $N_G(P)$ of $P$ in $G$ are derived equivalent
(Rickard equivalent).
This conjecture is called 
{\sl Strong Version of \sl Brou{\'e}'s Abelian Defect Group Conjecture}.
In this paper, we prove that the strong version of 
Brou{\'e}'s abelian defect group conjecture
is true for the non-principal $2$-block $A$ with an elementary abelian
defect group $P$ of order $8$ of the sporadic simple Conway 
group {\sf Co}$_3$. This result completes the verification of the strong
version of Brou{\'e}'s abelian defect group conjecture for all primes $p$
and for all $p$-blocks 
of {\sf Co}$_3$.

\smallskip\noindent
{\it Keywords:} Brou{\'e}'s conjecture; abelian defect group;
sporadic simple Conway group}

\medskip \hrule

\bigskip

\section{Introduction and notation}\label{intro}

\noindent
In the representation theory of finite groups, one of the
most important and interesting problems
is to give an affirmative answer to a conjecture
which was introduced by Brou{\'e} around 1988
\cite{Broue1990},
and is nowadays called {\sl Brou{\'e}'s Abelian
Defect Group Conjecture}.
He actually conjectures the following:

\begin{Conjecture} 
[Strong version of Brou\'e's Abelian Defect Group Conjecture
 \cite
{Broue1990}, 
 \cite
{KoenigZimmermann}]
\label{ADGC}
Let $p$ be a prime, and let $(\mathcal K, \mathcal O, k)$ be a
splitting $p$-modular system for all subgroups of a
finite group $G$. Assume that $A$ is a block algebra of
$\mathcal OG$ with a defect group $P$ and that $A_N$ is a
block algebra of $\mathcal ON_G(P)$ such that $A_N$ is the
Brauer correspondent of $A$, where $N_G(P)$ is the normaliser of
$P$ in $G$. Then $A$ and $A_N$ should be derived equivalent
(Rickard equivalent) provided $P$ is abelian.
\end{Conjecture}

\noindent
In fact, a stronger conclusion than {\bf\ref{ADGC}} is expected,
namely that $A$ and $A_N$ are 
{\sl splendidly Rickard equivalent} in the sense
of Linckelmann (\cite{Linckelmann1998}, \cite{Linckelmann2001}), which
he calls {\sl splendidly derived equivalent},
see {\bf\ref{equiv}}.
Note that for principal block algebras, 
this notion coincides with the splendid equivalence 
given by Rickard in \cite{Rickard1996}.

\begin{Conjecture}
[Rickard \cite{Rickard1996}, 
\cite
{Rickard1998}]
\label{RickardConjecture}
Keeping the notation, we suppose that $P$ is abelian as in {\bf\ref{ADGC}}.
Then there should be a splendid Rickard equivalence between
the block algebras $A$ of $\mathcal OG$ and $A_N$ of $\mathcal ON_G(P)$.
\end{Conjecture}

\noindent
There are several cases where the conjectures
{\bf\ref{ADGC}} and {\bf\ref{RickardConjecture}}
of Brou{\'e} and Rickard, respectively, have been verified,
albeit the general conjecture is widely open;
for an overview, containing suitable references, 
see \cite{ChuangRickard}. As for general results
concerning blocks with a fixed defect group,
by \cite{Linckelmann1991}, \cite{Rickard1989},
\cite{Rouquier1995}, and \cite{Rouquier1998} the conjectures are
proved for blocks with cyclic defect groups in arbitrary 
characteristic; in characteristic $2$, 
by \cite{Linckelmann1994a}, \cite{Linckelmann1994b},
\cite{Rickard1996}, and \cite{Rouquier2001}
they are known to hold for blocks with elementary abelian
defect groups of order $4$, but already the case of
elementary abelian defect groups of order $8$ is open in general.
At least for {\sl principal} blocks in characteristic $2$
it has been already known (at least for experts) that
{\bf\ref{ADGC}} and {\bf\ref{RickardConjecture}} hold
by using a lifting method \cite[9.1.(3)]{Marcus},
and recently a new lifting method was found
\cite[Theorem 4.33]{CraRou}.

In the present paper we look at the case where a non-principal
block $A$ has an elementary abelian defect group $P$ of
order $8$, namely, $P = C_2 \times C_2 \times C_2$.
The numbers of irreducible ordinary 
characters $k(A)$ 
and of irreducible Brauer characters $\ell(A)$,
respectively, are important in block theory.
For the principal $2$-blocks they have been known for some time,
see \cite{Koshitani1980} and \cite{Landrock1981}, for instance.
However, only recently, the numbers 
of irreducible ordinary characters $k(A)$ 
and of irreducible Brauer characters $\ell(A)$ 
for non-principal $2$-blocks
have been determined in general,
see \cite{KessarKoshitaniLinckelmann2010}.
In \cite{KessarKoshitaniLinckelmann2010}
it is proved with the help of the classification of
finite simple groups, 
that Alperin's weight conjecture and also the
{\sl weak} version ({\sl character theoretic} version) 
of Brou{\'e}'s abelian defect group conjecture 
for arbitrary $2$-blocks with defect group $C_2 \times C_2 \times C_2$
are both true.
The {\sl strong} version of 
Brou{\'e}'s abelian defect group conjecture, namely,
the existence of Rickard splendid equivalences between blocks
corresponding via the Brauer correspondence
for arbitrary $2$-blocks with defect group 
$C_2 \times C_2 \times C_2$, is still open.
There are four cases for the inertial index $e$ of $A$
with the defect group $P = C_2 \times C_2 \times C_2$.
Namely, $e = 1, 3, 7$ or $21$,
since ${\mathrm{Aut}}(P) \cong \GL_3(2)$ has a unique
maximal $2'$-subgroup, up to conjugacy, which is isomorphic to
the Frobenius group $F_{21} = C_7 \rtimes C_3$ of order $21$.
For the cases where $e = 1$ everything is known because
the blocks are nilpotent, 
see Brou{\'e}-Puig \cite{BrouePuig1980-Inv}. 
For the case $e = 3$, there are results 
of Landrock \cite{Landrock1981} and Watanabe \cite{Watanabe}. 

Our objective in this paper now is to investigate
a non-principal $2$-block with elementary abelian 
defect group $P$ of order $8$, which has inertial index $21$.
An interesting candidate for this endeavour is the non-principal $2$-block
of Conway's third group ${\sf Co}_3$, for which we investigate 
whether the {\sl strong} version 
of Brou{\'e}'s abelian defect group conjecture holds;
for previous results on $\mathsf{Co}_3$, its defect groups, and $2$-modular
characters confer \cite[p.193 Table 6]{Fendel},
\cite[\S 7 p.1879]{Landrock1978},
\cite[Theorems 3.10 and 3.11]{Landrock1981},
and \cite{SuleimanWilson}, for example.
We remark that, as far as the quasi-simple groups related to the
sporadic simple groups are concerned, this is the only $2$-block 
for which Brou{\'e}'s abelian defect group 
conjecture is not yet known to hold, since
within this class of groups all other abelian $2$-blocks
are either cyclic or of Klein four defect, see \cite{Noeske}.

\medskip\noindent
Our main theorem of this paper is the following:

\begin{Theorem}
\label{MainTheorem}
Let $G$ be the sporadic simple Conway group ${\sf Co}_3$, and 
let $(\mathcal K, \mathcal O, k)$ be a splitting $2$-modular
system for all subgroups of $G$, 
see {\bf\ref{notation}}. 
Suppose that $A$ is a non-principal block algebra of $\mathcal OG$
with a defect group $P$ which is an elementary abelian group
of order $8$, and that $A_N$ is a block algebra
of $\mathcal ON_G(P)$ such that $A_N$ is the Brauer correspondent
of $A$. Then $A$ and $A_N$ are splendidly Rickard
equivalent, and hence the conjectures 
{\bf\ref{ADGC}} and {\bf\ref{RickardConjecture}} 
of Brou{\'e} and Rickard both hold.
\end{Theorem}

Actually, {\bf\ref{MainTheorem}} is the last tile in the mosaic
proving both Brou\'e's abelian defect group conjecture and Rickard's
conjecture for $\mathsf{Co}_3$ in arbitrary characteristic. 
Since $|G|
= 2^{10}{\cdot}3^7{\cdot}5^3{\cdot}7{\cdot}11{\cdot}23$, see
\cite[p.134]{Atlas}, as the conjectures are proved for blocks with cyclic
defect groups, it is sufficient to consider the primes $p \in
\{ 2, 3, 5 \}$. For odd $p$ the only block with defect at least 2
is the principal block, whose defect groups are not abelian. For
$p = 2$ there is precisely a unique block with 
a non-cyclic abelian defect group.
Its defect group is isomorphic to $C_2\times C_2\times C_2$ (see
\cite[$\mathsf{Co}_3$]{ModularAtlasProject}, 
\cite[p.1879]{Landrock1978} and 
\cite[p.494 \S 2]{SuleimanWilson}).
Therefore we may state the following immediate consequence of
{\bf\ref{MainTheorem}}:

\begin{Corollary}\label{ADGC-Co3}
The strong version of Brou\'e's abelian defect group conjecture
{\bf\ref{ADGC}} and
even Rickard's splendid
equivalence conjecture
{\bf\ref{RickardConjecture}} are true
for all primes $p$ and for all block algebras of $\mathcal OG$
if $G = {\sf Co}_3$.
\end{Corollary}

\noindent
As a matter of fact, 
the main result {\bf\ref{MainTheorem}} is obtained
by proving the following:

\begin{Theorem}\label{2ndMainTheorem} 
We keep the notation and the assumption as in 
{\bf\ref{MainTheorem}}.
Let $H$ be a maximal subgroup of $G$ with
$H = R(3) \times \mathfrak S_3 \geqslant N_G(P)$,
where $R(3) = {^2}{G_2(3)} \cong {\mathrm{SL}}_2(8) \rtimes C_3$
is the smallest Ree group,
$\mathfrak S_3$ is the symmetric group on $3$ letters,
and $C_3$ is the cyclic group of order $3$.
Let $B$ be a block algebra of $\mathcal OH$ such that
$B$ is the Brauer correspondent of $A$,
see \cite[Chap.5 Theorem 3.8]{NagaoTsushima}.
In addition, let $\mathfrak f$ denote the Green correspondence
with respect to $(G \times G, \Delta P, G \times H)$,
and let $M = \mathfrak f(A)$.
Then $M$ induces a Morita equivalence between $A$ and $B$,
and hence it is a Puig equivalence.
\end{Theorem}

\noindent
The following result is used
to get {\bf\ref{corollary-R(q)}} from
our main result {\bf\ref{2ndMainTheorem}}. 

\begin{Theorem} 
[Landrock-Michler \cite{LandrockMichler1980}
and Okuyama \cite{Okuyama1997}]
\label{LandrockMichlerOkuyama}
Let $p = 2$, and let $R(q) = {^2}{G_2}(q)$ be a Ree group,
where $q = 3^{2n+1}$ for some $n = 0, 1, 2, \cdots$.
Let $(\mathcal K, \mathcal O, k)$ be a splitting $2$-modular
system for all subgroups of $R(q)$, for all $q$ at the same time,
see \cite[Theorem 3.6]{Willems},
and let $B_0(\mathcal O R(q))$ be the principal block algebra
of the group algebra $\mathcal OR(q)$.
Then the block algebras $B_0(\mathcal O R(3))$ 
and $B_0(\mathcal O R(q))$ are Puig equivalent.
In particular, Brou{\'e}'s abelian defect group conjecture {\bf\ref{ADGC}}
and Rickard's conjecture {\bf\ref{RickardConjecture}}
hold for the principal block algebras of $R(q)$
for any $q$.
\end{Theorem} 

\begin{proof}
This follows from \cite[Theorem 5.3]{LandrockMichler1980}
and \cite[Example 3.3 and Remark 3.4]{Okuyama1997}.
\end{proof}

\begin{Corollary}\label{corollary-R(q)}
We keep the notation and the assumption as in {\bf\ref{MainTheorem}}. 
Let $R(q) = {^2}{G_2}(q)$ be a Ree group,
where $q = 3^{2n+1}$ for some $n = 0, 1, 2, \cdots$.
We may assume that $(\mathcal K, \mathcal O, k)$ also is a
splitting $2$-modular system for all subgroups of $R(q)$,
for all $q$ at the same time.
Let $B_0(\mathcal OR(q))$ be the principal block algebra
of the group algebra $\mathcal OR(q)$.
Then $A$ and $B_0(\mathcal O R(q))$ are Puig equivalent.
\end{Corollary}

\begin{Strategy}
Our starting point for this work is the observation that the
$2$-decomposition matrix for the non-principal block $A$ of 
${\sf Co}_3$ with an elementary abelian defect group of order $8$,
see \cite{SuleimanWilson}, is exactly the same as that for
the principal $2$-block $B$ of $R(3) \cong \SL_2(8) \rtimes C_3$, see
\cite{LandrockMichler1980}. Therefore it is natural to ask whether
these two $2$-block algebras are Morita equivalent not only over
an algebraically closed field $k$ of characteristic $2$ but also
over a complete discrete valuation ring $\mathcal O$ whose residue
field is $k$. Furthermore, one might even expect that they are
{\sl Puig equivalent}, see {\bf\ref{equiv}}. If this is the case,
since the two conjectures of Brou\'e and Rickard {\bf\ref{ADGC}}
and {\bf\ref{RickardConjecture}} respectively have been shown to
hold for the principal $2$-block of $R(3)$ in a paper of Okuyama
\cite{Okuyama1997}, it follows that these conjectures also hold for the
non-principal $2$-block of ${\sf Co}_3$ with the same defect group 
$P = C_2 \times C_2 \times C_2$.

The verification that $A$ and $B$ are indeed Morita equivalent
relies on theorems by Linckelmann, Brou{\'e}, Rickard and Rouquier. 
Linckelmann has shown in
\cite{Linckelmann1996MathZ} that a stable equivalence of Morita type
between $A$ and $B$ which maps simple modules to simple modules is in
fact a Morita equivalence, see {\bf{\ref{StableIndecomposableMorita}}}.
To obtain an appropriate stable equivalence, we employ a variant of a
"gluing" theorem, which is due to (originally Brou{\'e} 
\cite[6.3.Theorem]{Broue1994}), Rickard \cite[Theorem 4.1]{Rickard1996}, 
Rouquier \cite[Theorems 5.6 and 6.3, Remark 6.4]{Rouquier2001}, and
Linckelmann, see \cite{Linckelmann1998}, \cite{Linckelmann2009} 
and {\bf{\ref{gluing}}}: 
A stable equivalence between
two blocks $A$ and $B$ 
may be 
deduced
from Morita equivalences between unique
blocks of the centralisers of 
non-trivial subgroups of $P$ in 
${\sf Co}_3$ and $R(3)$.
Once we have obtained a stable equivalence of Morita type between $A$ and
$B$, it remains to show that it preserves simplicity of
modules as stated above.
Usually this may be a very hard task.
\end{Strategy}

\begin{Contents}
The paper is structured as follows:
In Section \ref{pre}, we give the fundamental lemmas which are used 
to prove our main results. Furthermore, we establish some properties
of the stable equivalences we consider, 
and collect some further results on Morita equivalences and Green
correspondence for ease of reference.
In Section \ref{Non-principal2-blocks}
we investigate non-principal $2$-blocks of the symmetric group
$\mathfrak S_5$ and the Mathieu group $\sf{M}_{12}$
whose structure will be used later on in order to get our
main theorems.
In Section \ref{co3} the main objective is to construct the stable
equivalence of Morita type between the blocks $A$ and $B$ as outlined
above. In order to apply gluing theorems of
Rouquier and Linckelmann \textbf{\ref{gluing}}, 
we begin by analysing the 2-local structure of
$\mathsf{Co}_3$ to identify the groups.
Then, 
we combine this knowledge and
what we get already in Section \ref{Non-principal2-blocks}
to give a stable equivalence $F$ as saught.
Section \ref{blocks} prepares the proof that 
$F$ maps simple $A$-modules to
simple $B$-modules. In order to prove
this fact, we collect information on simple and indecomposable modules in
the three blocks $A$, $B$, and $A_N$.
In Section~\ref{img} we determine the $F$-images of the simple $A$-modules,
thus showing that they are indeed all simple. 
Finally, in Section \ref{proof} we combine the previous results to give
complete proofs of our main theorems 
{\bf\ref{MainTheorem}},
{\bf\ref{ADGC-Co3}},
{\bf\ref{2ndMainTheorem}} and
{\bf\ref{corollary-R(q)}}.
At the end of the paper, we have collected several useful properties of the
stable equivalences obtained through {\bf\ref{gluing}}.
\end{Contents}

\begin{Computations}
A few words on computer calculations are in order.
To find our results, next to theoretical reasoning
we have to rely fairly heavily on computations.
Of course, many of the data contained in explicit
libraries and databases are of computational nature, 
and quite a few traces of further computer calculations 
are still left in the present exposition. But we would 
like to point out that we have found many
of our intermediate results by explicit computations first, 
which have subsequently been replaced by more theoretical arguments.

As tools, we use the computer algebra system {\sf GAP} \cite{GAP},
to calculate with permutation groups and tables of marks,
as well as with ordinary and Brauer characters. 
We also make use of the data library 
\cite{CTblLib}, in particular allowing for easy access to the data
compiled in \cite{Atlas}, \cite{ModularAtlas} and 
\cite{ModularAtlasProject},
and of the interface \cite{AtlasRepresentation} to the data library 
\cite{ModularAtlasRepresentation}.
Moreover, we use the computer algebra system {\sf MeatAxe} \cite{MA}
to handle matrix representations over finite fields,
as well as its extensions to compute 
submodule lattices \cite{LuxMueRin},
radical and socle series \cite{LuxWie},
homomorphism spaces and endomorphism rings \cite{LuxSzoke},
and direct sum decompositions \cite{LuxSzokeII}.
We give more comments later on where necessary.
\end{Computations}

\begin{Notation}\label{notation} 
Throughout this paper, we use the standard 
notation and terminology as is used in
\cite{NagaoTsushima}, \cite{Thevenaz} and \cite{Atlas}.

Let $k$ be a field and assume that $A$ and $B$ are
finite dimensional $k$-algebras. 
We denote by $\mathrm{mod}{\text -}A$,
$A{\text -}\mathrm{mod}$ and
$A{\text -}\mathrm{mod}{\text -}B$
the categories of finitely generated right $A$-modules,
left $A$-modules and $(A,B)$-bimodules, respectively.
We write $M_A$, $_AM$ and $_AM_B$ when $M$ is 
a right $A$-module,
a left $A$-module and an $(A,B)$-bimodule. 
In this note, a module always refers to 
a finitely generated right module, unless stated otherwise.
We let $M^\vee =\Hom_A(M_A, A_A)$ be the $A$-dual
of $M$, so that $M^\vee$ becomes a left $A$-module 
via $(a\phi)(m) = a{\cdot}\phi(m)$ for $a \in A$, $\phi \in M^\vee$
and $m \in M$,
and we let $M^\circledast =\Hom_k(M, k)$ 
be the $k$-dual of $M$, so that $M^\circledast$
becomes a left $A$-module as well
via $(a\phi)(m) = \phi(ma)$ 
for $a \in A$, $\phi \in M^\circledast$ and $m \in M$.
For $A$-modules $M$ and $N$ we write $[M,N]^A$ for 
$\mathrm{dim}_k[\mathrm{Hom}_{A}(M,N)]$.
We fix for a while an $A$-module $M$.
Then, for a projective cover $P(S)$ of a simple $A$-module $S$,
we write $[P(S) \mid M]^A$ for the multiplicity of direct summands
of $M$ which are isomorphic to $P(S)$.
We denote by ${\mathrm{soc}}(M)$ and ${\mathrm{rad}}(M)$
the socle and the radical of $M$, respectively,
and hence ${\mathrm{rad}}(M) = M{\cdot}{\mathrm{rad}}(A)$.
For simple $A$-modules $S_1, \cdots, S_n$, 
and positive integers $a_1, \cdots, a_n$,
we write that 
"$M = a_1 \times S_1 + \cdots + a_n \times S_n$,
as composition factors"
when the set of all composition factors are
$a_1$ times $S_1$, $\cdots$, $a_n$ times $S_n$.
In order to avoid being ambiguous, we sometimes
use convention such as
$M = a_1 \times [S_1] + \cdots + a_n \times [S_n]$.
For another $A$-module $L$, we write $M | L$ when
$M$ is isomorphic to a direct summand of $L$ as
an $A$-module.
If $A$ is self-injective,
the stable module category $\underline{\mathrm{mod}}{\text -}A$,
is the quotient category of $\mathrm{mod}{\text -}A$
with respect to the projective $A$-homomorphisms, that is
those factoring through a projective module.

In this paper, $G$ is always a finite group and we fix a
prime number $p$. Assume that $(\mathcal K, \mathcal O, k)$ is a
splitting $p$-modular system for all subgroups of $G$, that is
to say, $\mathcal O$ is a complete discrete valuation ring of
rank one such that its quotient field is $\mathcal K$ which is
of characteristic zero, and its residue field
$\mathcal O/\mathrm{rad}(\mathcal O)$ is $k$, which is of
characteristic $p$, and that $\mathcal K$ and $k$ are
splitting fields for all subgroups of $G$.
By an $\mathcal OG$-lattice we mean a finitely generated
right $\mathcal OG$-module which is a free $\mathcal O$-module.
We denote by $k_G$ the trivial $kG$-module, 
and similarly by $\mathcal{O}_G$
the trivial $\mathcal{O}G$-lattice.
If $X$ is a $kG$-module, then we write $X^* =\mathrm{Hom}_k(X,k)$ 
for the {\sl contragredient} of $X$, 
namely, $X^*$ 
is again a right $kG$-module
via $(\varphi g) (x) = \varphi(xg^{-1})$ for $x \in X$,
$\varphi \in X^*$ and $g \in G$;
if no confusion may arise we also call this the {\sl dual} of $X$.
Let $H$ be a subgroup of $G$, and let $M$ and $N$ be
an $\mathcal OG$-lattice and an $\mathcal OH$-lattice, respectively.
Then let ${M}{\downarrow}^G_H = {M}{\downarrow}_H$ be the
restriction of $M$ to $H$, and let 
${N}{\uparrow}_H^G = {N}{\uparrow}^G = 
(N \otimes_{\mathcal OH}\mathcal OG)_{\mathcal OG}$ 
be the induction (induced module) of $N$ to $G$.
A similar definition holds for $kG$- and $kH$-modules. 
For a subgroup $Q$ of $G$ we write
Scott$(G,Q)$ for the (Alperin-)Scott module with respect to
$Q$ in $G$, see \cite[Chap.4 p.297]{NagaoTsushima}.

We denote by $\mathrm{Irr}(G)$ and $\mathrm{IBr}(G)$ the sets of
all irreducible ordinary and Brauer characters of $G$, respectively.
Since the character field $\Q(\chi):=\Q(\chi(g)\:;\:g\in G)\subseteq\mathcal K$ 
of any character $\chi\in\mathrm{Irr}(G)$
is contained in a cyclotomic field, we may identify $\Q(\chi)$
with a subfield of the complex number field $\C$, hence we may think
of characters having values in $\C$. In particular,
we write $\chi^{*}$ for the complex conjugate of $\chi$,
where of course $\chi^{*}$ is the character of the 
$\mathcal K G$-module contragredient to the $\mathcal K G$-module
affording $\chi$.
For $\chi,\psi\in\mathrm{Irr}(G)$
we denote by $(\chi, \psi)^G$ the usual inner product.
If $A$ is a block algebra ($p$-block) of $\mathcal OG$,
then we write $\mathrm{Irr}(A)$ and $\mathrm{IBr}(A)$ for
the sets of all characters in $\mathrm{Irr}(G)$ and $\mathrm{IBr}(G)$
which belong to $A$, respectively.
We denote by $B_0(kG)$ the principal block algebra of $kG$,
we write $1_G$ for the trivial character of $G$.

Let $G'$ be another finite group, and let $V$ be an 
$(\mathcal OG, \mathcal OG')$-bimodule. 
Then we can regard $V$ as a right
$\mathcal O[G \times G']$-module.
A similar definition holds for $(kG, kG')$-bimodules.
We denote by $\Delta G= \{ (g,g) \in G \times G \, | \, g \in G \}$
the diagonal copy of $G$ in $G \times G$.
For an $(\mathcal OG, \mathcal OG')$-bimodule $V$ and a common
subgroup $Q$ of $G$ and $G'$, we set
$V^{\Delta Q}  
= \{ v \in V \ | \ qv = vq \text{ for all } q \in Q \}$.
If $Q$ is a $p$-group, the Brauer 
construction is defined to be
the quotient
$V(\Delta Q) = V^{\Delta Q}/ [\sum_{R \lneqq Q}
 {\mathrm{Tr}}{\uparrow}_{R}^{Q}(V^{\Delta R}) 
 + {\mathrm{rad}}{\mathcal O}{\cdot}V^{\Delta Q}]$,
where ${\mathrm{Tr}}{\uparrow}_{R}^{Q}$ 
is the usual trace map. The Brauer homomorphism
${\mathrm{Br}}_{\Delta Q}: (\mathcal OG)^{\Delta Q} 
 \rightarrow  kC_G(Q)$
is obtained from composing the canonical epimorphism
$(\mathcal OG)^{\Delta Q} \twoheadrightarrow
 (\mathcal OG)(\Delta Q)$
and the canonical isomorphism
$(\mathcal OG)(\Delta Q) 
 \overset{\approx}{\rightarrow} kC_G(Q)$.

Let $n$ be a positive integer. Then, 
$\mathfrak A_n$ and $\mathfrak S_n$ denote 
the alternating and the symmetric groups on $n$ letters. 
Also, $C_n$ and $D_{2n}$ 
denote the cyclic group of order $n$ and the dihedral group
of order $2n$, respectively.
Moreover, for $i \in \{ 10, 11, 12, 22, 23, 24\}$, 
$\mathsf{M}_i$ denotes the Mathieu group of degree $i$.
We denote by $Z(G)$ the centre of $G$, and by $S^g$
a set $g^{-1}Sg$ for $g \in G$ and a subset $S$ of $G$.

\end{Notation}

\begin{Equivalences}\label{equiv}
Let $A$ and $A'$ be block algebras of $\mathcal OG$ and $\mathcal OG'$, 
respectively. Then we say that $A$ and $A'$ are {\sl Puig equivalent}
if $A$ and $A'$ have a common defect group $P$, 
and if there is a Morita equivalence between $A$ and $A'$
which is induced by an $(A,A')$-bimodule $\mathfrak M$ 
such that, as a right
$\mathcal O[G \times G']$-module, 
$\mathfrak M$ is a trivial source
module and $\Delta P$-projective. A similar definition
holds for blocks of $kG$ and $kG'$.
Due to a result of Puig (and independently of Scott), 
see \cite[Remark 7.5]{Puig1999},
this is equivalent to a condition that
$A$ and $A'$ have source algebras which are isomorphic as
interior $P$-algebras,
see \cite[Theorem 4.1]{Linckelmann2001}.

We say that $A$ and $A'$ are 
{\sl stably equivalent of Morita type} 
if there exists an $(A, A')$-bimodule 
$\mathfrak M$ such that
${_A}\mathfrak M$ is projective as a left $A$-module,
$\mathfrak M_{A'}$ is projective as a right $A'$-module,
$_A(\mathfrak M \otimes_{A'} \mathfrak M^\vee)_A 
  \cong {_A}{A}{_A} \oplus
 ({\mathrm{proj}} \ (A,A){\text{-}}{\mathrm{bimod}})$
and
$_{A'} (\mathfrak M^\vee \otimes_A \mathfrak M)_{A'} 
\cong {_{A'}}{A'}{_{A'}} \oplus
 ({\mathrm{proj}} \ (A' ,A'){\text{-}}{\mathrm{bimod}})$.

We say that $A$ and $A'$ are 
{\sl splendidly stably equivalent of Morita type}
if $A$ and $A'$ have a common defect group $P$ and 
the stable 
equivalence of Morita type is induced by
an $(A,A')$-bimodule $\mathfrak M$ 
which is a trivial source
$\mathcal O[G \times G']$-module and is $\Delta P$-projective,
see \cite[Theorem~3.1]{Linckelmann2001}.

We say that $A$ and $A'$ are 
{\sl derived equivalent (or Rickard equivalent)} if
${\mathrm{D}}^b({\mathrm{mod}}{\text{-}}A)$ and
${\mathrm{D}}^b({\mathrm{mod}}{\text{-}}A')$ are
equivalent as triangulated categories,
where 
${\mathrm{D}}^b({\mathrm{mod}}{\text{-}}A)$
is the bounded derived category of ${\mathrm{mod}}{\text{-}}A$.
In that case, there even is a {\sl Rickard complex}
$M^{\bullet} \in
{\mathrm{C}}^b (A{\text{-}}{\mathrm{mod}}{\text{-}}A')$,
where the latter 
is the category of bounded complexes
of finitely generated $(A,A')$-bimodules,
all of whose terms are projective both as 
left $A$-modules and as right $A'$-modules, such that 
$M^{\bullet}\otimes_{A'}(M^{\bullet})^{\vee}\cong A$
in $K^b(A{\text{-}}{\mathrm{mod}}{\text{-}}A)$
and
$(M^{\bullet})^{\vee}\otimes_A M^{\bullet}\cong A'$
in $K^b(A'{\text{-}}{\mathrm{mod}}{\text{-}}A')$,
where $K^b(A{\text{-}}{\mathrm{mod}}{\text{-}}A)$ is the homotopy
category associated with 
${\mathrm{C}}^b (A{\text{-}}{\mathrm{mod}}{\text{-}}A)$.
In other words, in that case we even have 
$K^b({\mathrm{mod}}{\text{-}}A) \cong 
K^b({\mathrm{mod}}{\text{-}}A')$. 

We say that $A$ and $A'$ are 
{\sl splendidly Rickard equivalent}
if $K^b({\mathrm{mod}}{\text{-}}A)$ and 
$K^b({\mathrm{mod}}{\text{-}}A')$ are equivalent
via a Rickard complex 
$M^{\bullet}\in 
{\mathrm{C}}^b (A{\text{-}}{\mathrm{mod}}{\text{-}}A')$
as above,
such that additionally each of its terms is a 
direct sum of $\Delta P$-projective 
trivial source modules as an
$\mathcal O[G \times G']$-module.
\end{Equivalences}

\section{Preliminaries}\label{pre}

In this section we give several theorems crucial to the later sections
of this paper. We state these results in a more general context; in
particular, $G$ is an arbitrary finite group and $(K,\mathcal{O},k)$ is
a $p$-modular splitting system for $G$. As we draw upon these lemmas
frequently in the sequel, we state these explicitly for the convenience
of the reader and ease of reference.

As stated in the introduction, our approach centres around
\textbf{\ref{StableIndecomposableMorita}} which allows us to verify that a
stable equivalence of Morita type 
is in fact a Morita equivalence. The stable 
equivalences investigated are obtained with the help of \textbf{\ref{gluing}},
and are realised by tensoring with a bimodule given through Green
correspondence. We proceed to study several properties of these
stable equivalences, and give some further results needed in the upcoming
parts of this paper. We refer the reader also to the appendix for
a more detailed discussion of further properties of stable equivalences
obtained through {\bf\ref{gluing}}.

\begin{Lemma} [Linckelmann \cite{Linckelmann1996MathZ}] 
\label{StableIndecomposableMorita}
Let $A$ and $B$ be finite-dimensional $k$-algebras
such that $A$ and $B$ are both
self-injective and indecomposable as algebras, but not simple.
Suppose that there is an 
$(A,B)$-bimodule $M$ such that $M$ induces a stable equivalence
between the algebras $A$ and $B$.

\begin{enumerate}
  \renewcommand{\labelenumi}{\rm{(\roman{enumi})}}
    \item
If $M$ is indecomposable then for any simple $A$-module $S$, the $B$-module
$(S \otimes_A M)_B$ is non-projective and indecomposable.
    \item
If for all simple $A$-module $S$ the $B$-module 
$S \otimes_A M$ is simple then $M$ induces a Morita
equivalence between $A$ and $B$.
    \item
If $(M, M^\vee)$ induces a stable equivalence of Morita type
between $A$ and $B$ then there is a unique (up to isomorphism)
non-projective indecomposable $(A,B)$-bimodule $M'$ such that
$M' \mid M$, and $(M', {M'}^\vee)$ induces a stable equivalence
of Morita type between the algebras $A$ and $B$. 
\end{enumerate}
\end{Lemma}

\begin{proof}
(i) and (ii) respectively are given in 
\cite[Theorem 2.1(ii) and (iii)]{Linckelmann1996MathZ}.
Part (iii) follows by 
\cite[Theorem 2.1(i) and Remark 2.7]{Linckelmann1996MathZ}.
\end{proof}

We obtain a suitable stable equivalence to apply
\textbf{\ref{StableIndecomposableMorita}} 
through a ``gluing theorem'' as given in
\textbf{\ref{gluing}}.

\begin{Lemma}[Koshitani-Linckelmann \cite{KoshitaniLinckelmann2005}]
\label{KoshitaniLinckelmann}
Let $A$ be a block algebra of $kG$ with defect group $P$,
and let $(P,e)$ be a maximal $A$-Brauer pair
such that $H=N_G(P,e) = N_G(P)$.
Let $B$ be a block algebra of $kH$ such that
$B$ is the Brauer correspondent of $A$.
Let $\mathfrak f$ be the Green correspondence with respect to
$(G \times G, \Delta P, G \times H)$, 
and set $M = \mathfrak f(A)$, in particular
$M$ is an indecomposable $(A,B)$-bimodule with vertex $\Delta P$.

Take any subgroup $Q$ of $Z(P)$, and set
$G_Q = C_G(Q)$ and $H_Q = C_H(Q)$.
Let $e_Q$ and $f_Q$ be block idempotents of
$kG_Q$ and $kH_Q$ satisfying 
$(Q,e_Q) \subseteq (P,e)$ and $(Q,f_Q) \subseteq (P,e)$, respectively,
see \cite[(40.9)~Corollary]{Thevenaz}.
Let $\mathfrak f_Q$ be the Green correspondence with respect to
$(G_Q \times G_Q, \Delta P, G_Q \times H_Q)$.
Then we have
\[ 
e_Q M(\Delta Q) f_Q =
   \mathfrak f_Q  (e_Q kG_Q)
\]
and this is a unique (up to isomorphism) indecomposable
direct summand of 
$(e_Q kG_Q){\downarrow}_{G_Q \times H_Q}$ with vertex $\Delta P$.
\end{Lemma}

\begin{proof}
We know $M = \mathfrak f(A) \mid {A}{\downarrow}^{G \times G}_{G \times
H} \mid {_{kG}}{kG}_{kH}$. Hence, $M(\Delta Q) \mid (kG)(\Delta Q) =
kC_G(Q) = kG_Q$.
Thus,
\[
e_Q M(\Delta Q) f_Q \,\big|\,
e_Q kG_Q f_Q \,\big| \,
{(e_Q kG_Q)}{\downarrow}^{G_Q \times G_Q}_{G_Q \times H_Q}.
\]
By \cite[Theorem]{KoshitaniLinckelmann2005},
$e_Q M(\Delta Q) f_Q$ is an indecomposable 
$k[G_Q \times H_Q]$-module with vertex $\Delta P$.
Thus Green correspondence
yields $e_Q M(\Delta Q) f_Q = \mathfrak f_Q  (e_Q kG_Q)$.
\end{proof}

\begin{Lemma}[Linckelmann \cite{Linckelmann2001}, \cite{Linckelmann2009}]
\label{gluing}
Let $A$ be a block algebra of $\mathcal OG$ with a defect group $P$,
and let $(P,e)$ be a maximal $A$-Brauer pair in $G$.
Set 
$H = N_G(P,e)$, 
Assume that
\begin{enumerate}
 \renewcommand{\labelenumi}{\rm{(\arabic{enumi})}}
 \item $P$ is abelian,
 \item for each $Q$ with $1 \, \not= \, Q \, \leqslant P$,
  $kC_G(Q)$ has a unique block algebra $A_Q$ with the defect group $P$,
 \item for each $Q$ with $1 \, \not= \, Q \, \leqslant P$,
  $kC_H(Q)$ has a unique block algebra $B_Q$ with the defect group $P$.
\end{enumerate}
Let $B$ a block algebra of $\mathcal OH$ which is the Brauer
correspondent of $A$. For each subgroup $Q$ of $P$, let $e_Q$ and $f_Q$
be the block idempotents of $A_Q$ and $B_Q$, respectively, and hence
$A_Q = kC_G(Q)e_Q$ and $B_Q = kC_H(Q)f_Q$. Note that $e_P = e = f_P$ and
$A_P = B_P$. Let $\mathfrak f$ be the Green correspondence with respect
to $(G \times G, \Delta P, G \times H)$, and set ${_A}M_B = \mathfrak
f(A)$, see {\bf\ref{Green-GxG}}. Moreover, let ${\mathfrak f}_Q$ be the
Green correspondence with respect to $(C_G(Q) \times C_G(Q), \Delta P,
C_G(Q) \times C_H(Q))$. Now, assume further that
\begin{enumerate}
 \renewcommand{\labelenumi}{\rm{(\arabic{enumi})}}
 \setcounter{enumi}{3}
 \item for each non-trivial proper subgroup $Q$ of $P$, 
  the $(A_Q, B_Q)$-bimodule $\mathfrak f_Q(A_Q)$
  induces a Morita equivalence between $A_Q$ and $B_Q$.
 \end{enumerate}
Then the $(A,B)$-bimodule $M$ induces a stable equivalence of Morita
type between $A$ and $B$.
\end{Lemma}

\begin{proof}
First, note $H = N_G(P)$.
Secondly, it follows from {\bf\ref{KoshitaniLinckelmann}} that 
$e_Q{\cdot}M(\Delta Q){\cdot}f_Q = \mathfrak f_Q(A_Q)$
for each $Q \leqslant P$ since $P$ is abelian by (1).
Then since $A_P = B_P$ and since
$A_P = \mathfrak f_P(A_P) = e{\cdot}M(\Delta P){\cdot}e$, 
the $(A_P, B_P)$-bimodule $e_P{\cdot}M(\Delta P){\cdot}e_P$
induces a Morita equivalence between $A_P$ and $B_P$.

Now, for each $Q \leqslant P$, 
it follows from the uniqueness of $e_Q$ and $f_Q$ 
that
\[ (Q,e_Q) \subseteq (P,e) \qquad {\text{and}} \qquad (Q,f_Q) \subseteq
(P,e).\]

Next, we want to claim
\[ E_G\Big( (Q, e_Q), (R, e_R)\Big) = E_H\Big( (Q, f_Q), (R, f_R)\Big)
\qquad {\text{for}} \ \ Q, \, R \leqslant P, \]
where
$E_G\Big( (Q, e_Q), (R, e_R)\Big)$ is the set $\{ \varphi : Q
\rightarrow R \ | \ {\text{there}} \ {\text{is}} \ g \in G \
{\text{with}} \ \varphi (u) = u^g, \text{ for all } u \in Q, \ \
{\text{and}} \ (Q, e_Q)^g \subseteq (R, e_R) \}$,
see \cite[p.821]{Linckelmann2001}.
This is known by using 
\cite[Proposition 4.21 and Theorem~3.4]{AlperinBroue} and 
\cite[Theorem 1.8(1)]{BrouePuig1980} since $P$
is abelian, see \cite[The proof of 1.15.~Lemma]{KoshitaniKunugiWaki2004}
for details. Therefore we can apply Linckelmann's result
\cite[Theorem 3.1]{Linckelmann2001}. 
\end{proof}
We remark that in \cite[Theorem 3.1]{Linckelmann2001} and 
\cite[Theorem A.1]{Linckelmann2009}, 
Linckelmann proves more general theorems than
{\bf\ref{gluing}}. However, we formulate with {\bf\ref{gluing}} 
a version which is specifically tailored to
our practical purposes, and use this ad hoc version in the sequel.

In the notation of \textbf{\ref{gluing}}, we have that the bimodule $M$
realising a stable equivalence between $A$ and $B$ is a Green
correspondent of $A$. In fact it is a direct summand of 
$1_A\cdot kG\cdot 1_B$ as the next lemma shows.

\begin{Lemma}\label{Green-GxG}
Let $A$ be a block algebra of $kG$ with defect group $P$.
Assume that $(P, e)$ is a maximal $A$-Brauer pair
such that $H=N_G(P,e)= N_G(P)$.
Let $B$ be a block algebra of $kH$ such that
$B$ is the Brauer correspondent of $A$.
Let $\mathfrak f$ be the Green correspondence with respect to 
$(G \times G, \Delta P, G \times H)$.
Then we have $\mathfrak f(A) \mid 1_A{\cdot}kG{\cdot}1_B$.
\end{Lemma}

\begin{proof}
It follows from 
\cite[Theorem 5(i)]{AlperinLinckelmannRouquier} that $({A}{\downarrow^{G
\times G}_{G \times H}}){\cdot}1_B = 1_A{\cdot}kG{\cdot}1_B$ has a
unique (up to isomorphism) indecomposable direct summand with vertex
$\Delta P$. Clearly, $1_A{\cdot}kG{\cdot}1_B \mid ({A}{\downarrow^{G
\times G}_{G \times H}})$, hence by Green correspondence we have
$\mathfrak f(A) \mid 1_A{\cdot}kG{\cdot}1_B$.
\end{proof}

We remark that a stable equivalence of Morita type induced by the Green
correspondent $\mathfrak{f}(A)$ in the context of {\bf\ref{Green-GxG}}
preserves vertices and sources,
and takes indecomposable modules to their Green correspondents,
see (i) and (iii) in {\bf\ref{SourceVertex}}.

\begin{Lemma}\label{twoStableEquivalences}
Let $G$, $H$, and $L$ be finite groups, all of which have a
common non-trivial $p$-subgroup $P$, and assume that $H \leqslant G$.
Let $A$, $B$, and $C$ be block algebras of
$kG$, $kH$, and $kL$, respectively, 
all of which have $P$ as their defect group.
In addition, suppose that
a pair $({_A}\mathfrak{M}_B, {_B}{\mathfrak{M}'}_A)$ induces a stable 
equivalence
between $A$ and $B$ such that
${_A}\mathfrak{M}_B \mid k_{\Delta P}{\uparrow}^{G \times H}$,
${_B}{\mathfrak{M}'}_A \mid k_{\Delta P}{\uparrow}^{H \times G}$
{\rm{(}}and hence $\mathfrak{M}$ and $\mathfrak{M}'$ 
preserve vertices and sources, 
see {\rm{(i)}} and {\rm{(iii)}} of {\bf\ref{SourceVertex}}{\rm{)}}.
Similarly, suppose that
a pair $({_B}\mathfrak{N}_C,{_C}{\mathfrak{N}'}_B)$ induces a stable 
equivalence between $B$ and $C$ such that
${_B}\mathfrak{N}_C \mid k_{\Delta P}{\uparrow}^{H \times L}$,
${_C}{\mathfrak{N}'}_B \mid k_{\Delta P}{\uparrow}^{L \times H}$
{\rm{(}}and hence 
$\mathfrak{N}$ and $\mathfrak{N}'$ preserve vertices and sources,
see {\rm{(i)}} and {\rm{(iii)}} of {\bf\ref{SourceVertex}}{\rm{)}}.
Then we have $(A,C)$- and $(C,A)$-bimodules
$M$ and $M'$, respectively, which
satisfy the following:
\begin{enumerate}
 \renewcommand{\labelenumi}{\rm{(\arabic{enumi})}}
   \item
    ${_A}(\mathfrak{M} \otimes_B \mathfrak{N})_C 
 = {_A}M_C 
   \oplus ({\mathrm{proj}} \ (A,C){\text{-}}{\mathrm{bimodule}})$
and
\newline
${_C}(\mathfrak{N}' \otimes_B{\mathfrak{M}'})_A = {_C}M'_A 
   \oplus ({\mathrm{proj}} \ (C,A){\text{-}}{\mathrm{bimodule}})$.
   \item
${_A}M_C$ and 
${_C}M'_A$ are both non-projective indecomposable.
   \item
The pair $(M, M')$ induces a stable
equivalence between $A$ and $C$.
   \item
The functors 
$$
- \otimes_A M: {\mathrm{mod}}{\text{-}}A
  \longrightarrow {\mathrm{mod}}{\text{-}}C
$$
and 
$$
- \otimes_C M': {\mathrm{mod}}{\text{-}}C
  \longrightarrow {\mathrm{mod}}{\text{-}}A
$$
preserve vertices and sources of indecomposable modules.
That is, for non-projective indecomposable $A{\text{-}}$ and
$C{\text{-}}$modules $X$ and $Y$ corresponding via
$X \otimes_A M = Y \oplus ({\mathrm{proj}})$
and
$Y \otimes_C M' = X \oplus ({\mathrm{proj}})$,
respectively,
there is a non-trivial $p$-subgroup $Q$ and an indecomposable
$kQ$-module $S$ such that $Q$ is a common vertex of $X$ and $Y$
and that $S$ is a common source of $X$ and $Y$.
   \item
$_A M_C \mid {k_{\Delta P}}{\uparrow}^{G \times L}$
and
$_C {M'}_A \mid {k_{\Delta P}}{\uparrow}^{L \times G}$.
   \item
    In particular, if a pair $(\mathfrak{M}, \mathfrak{M}^\vee)$
induces a stable equivalence of Morita type between $A$ and $B$,
and if a pair $(\mathfrak{N}, \mathfrak{N}^\vee)$
induces a stable equivalence of Morita type between $B$ and $C$,
then we can replace $M'$ above by $M^\vee$
and we have that the pair $(M, M^\vee)$ induces
a stable equivalence of Morita type between $A$ and $C$.
\end{enumerate}
\end{Lemma}

\begin{proof}
Obviously, the pair
$({_A}(\mathfrak{M} \otimes_B \mathfrak{N})_C, 
{_C}(\mathfrak{N}'\otimes_B \mathfrak{M}')_A)$
induces a stable equivalence between $A$ and $C$.
Clearly, ${_A}(\mathfrak{M}\otimes_B \mathfrak{N})$, 
$(\mathfrak{M}\otimes_B \mathfrak{N})_C$,
${_C}(\mathfrak{N}'\otimes_B \mathfrak{M}')$, 
and $(\mathfrak{N}'\otimes_B \mathfrak{M}')_A$ are all projective.
Since $A$ and $C$ are symmetric algebras, it follows from
{\bf\ref{StableIndecomposableMorita}}(iii)
that there are
$(A,C)$- and $(C,A)$-bimodules $M$ and $M'$
which satisfy the conditions (1)--(4).

Next we want to show (5). It follows from
\cite[Chap.5 Lemma~10.9(iii)]{NagaoTsushima} that
\begin{align*}
 M \, 
 &{\Big|} \, \mathfrak{M} \otimes_B \mathfrak{N} \, {\Big|} \,
       (k_{\Delta P}{\uparrow}^{G \times H}) \otimes_{kH}
       (k_{\Delta P}{\uparrow}^{H \times L})
\\
&\cong (kG \otimes_{kP}kH)\otimes_{kH}(kH \otimes_{kP}kL)
\cong
     kG \otimes_{kP}[(kH){\downarrow}^{H\times H}_{P\times P}]
         \otimes_{kP} kL 
\\
&\cong kG \otimes_{kP}
  \Big( \bigoplus _{h \in [P \backslash H/P]} k[PhP] \Big)\otimes_{kP}kL
\cong \bigoplus _{h \in [P \backslash H/P]}  
  k[PhP]{\uparrow}^{G\times L}_{P\times P}.
\end{align*}
Since ${_A}M_C$ is indecomposable, there is an element
$h \in H$ such that 
$M \mid k[PhP]{\uparrow}^{G\times L}_{P\times P}$.
Set $(P\times P)_h = 
\{ (u, h^{-1}uh) \in P \times P \mid 
    u \in P \cap hPh^{-1} \}$. 
Then
$$
(P\times P)_h 
= 
\{ (huh^{-1}, u) \in P \times P \mid 
    u \in P \cap h^{-1}Ph \}
= (h,1){\cdot}\Delta[P\cap P^h]{\cdot}(h^{-1}, 1).
$$
We get by \cite[Chap.5 Lemma~10.9(iii)]{NagaoTsushima}
that
$k[PhP] \cong k_{(h,1)\Delta[P \cap P^h](h^{-1},1)}
{\uparrow}^{P\times P}$,
and hence
$M \mid  
k_{(h,1)\Delta[P \cap P^h](h^{-1},1)}{\uparrow}^{G\times L}$.
Now, since $(h^{-1}, 1) \in H \times L \leqslant G \times L$,
we have that
$$
M \mid  k_{\Delta[P \cap P^h]}{\uparrow}^{G\times L}
\cong kG \otimes_{kQ} kL
$$
where $Q = P \cap P^h$. 
Then for any $X$ in ${\mathrm{mod}}{\text{-}}A$
the module $X\otimes_A M$ has a vertex contained in $Q$.
If $Q$ is a proper subgroup of $P$ then,
since $(M, M')$
induces a stable equivalence between $A$ and $C$,
any module in ${\mathrm{mod}}{\text{-}}C$
has a vertex properly contained in $P$, a contradiction since
$P$ is a defect group of $C$.
Hence $Q = P$, so that $h \in N_H(P) \subseteq N_G(P)$.
Therefore $M \mid k_{\Delta P}{\uparrow}^{G \times L}$.
An analogous argument gives the claim for $M'$.

(6) Follows from (1)--(5) and 
{\bf\ref{StableIndecomposableMorita}}(iii).
\end{proof}

Next, we give some results on Morita equivalences and tensor products, which
will be useful in Section~\ref{co3}.

\begin{Lemma}\label{TensorMorita}
The following hold:
\begin{enumerate}
 \renewcommand{\labelenumi}{\rm{(\roman{enumi})}}
 \item Let $A$, $B$, $C$ and $D$ be finite dimensional $k$-algebras.
  Assume that an $(A,B)$-bimodule $M$ realises a Morita equivalence
  between $A$ and $B$, and so does a $(C,D)$-bimodule $N$ between $C$
  and $D$. Then the $(A \otimes C, B \otimes D)$-bimodule $M \otimes N$
  induces a Morita equivalence between $A \otimes C$ and $B \otimes D$.
 \item Keep the notation as in {\rm{(i)}}.
  Assume that $P$ is a common $p$-subgroup of finite groups $G$ and $H$,
  and that $Q$ is a subgroup of $P$. Suppose moreover that $A$ and $B$
  respectively are block algebras of $kG$ and $kH$, $C = D = kQ$ and
  $N = {_{kQ}}kQ_{kQ}$. If a $(kG, kH)$-bimodule $M$ satisfies that $M
  \mid {k_{\Delta P}}{\uparrow}^{G \times H}$, then $(M \otimes N)\mid
  {k_{\Delta [P \times Q]}}{\uparrow}^{(G\times Q) \times (H \times
  Q)}$.
\end{enumerate}
\end{Lemma}

\begin{proof}
The proof of (i) is straightforward. For (ii) observe that $k_{\Delta
P}{\uparrow^{(G\times Q)\times(H\times Q)}}$ is isomorphic to $k[G\times
Q] \otimes_{k[P\times Q]} k[H \times Q]$, and hence to $(kG \otimes_{kP}
kH) \otimes kQ$ as $k[G \times Q] \otimes k[H\times Q]$-bimodules. The
latter is isomorphic to $k_{\Delta P}{\uparrow^{G\times H}} \otimes
{_{kQ}}kQ_{kQ}$.
\end{proof}

Note that we cannot replace the \textsl{Morita} equivalence in
{\bf\ref{TensorMorita}} by a \textsl{stable} equivalence in general,
see \cite[Question 3.8]{Rickard1998ICRA}.

\begin{Lemma}\label{TensorPuigEquivalence}
Let $G$ and $H$ be finite groups, let 
$A$ and $B$, respectively, be block algebras of
$kG$ and $kH$. Let $X$ be an indecomposable $kG$-module in $A$, 
and let $Y$ be an indecomposable $kH$-module in $B$.
Then the following hold:
\begin{enumerate}
\renewcommand{\labelenumi}{\rm{(\roman{enumi})}}
    \item
If $B$ is of defect zero, then
a block algebra $A \otimes B$ of $k[G \times H]$ is
Puig equivalent to $A$.
    \item
Set $Z = X \otimes Y$. Then $Z$ is an indecomposable
$k[G \times H]$-module in $A \otimes B$.
If $X$ and $Y$ are are trivial source modules, then 
$Z$ is a trivial source module as well.
    \item
If $Y$ is projective, and $Q$ is a vertex of $X$,
then $Q \times \langle 1 \rangle$ is a vertex of $Z$,
and $Z$ is a trivial source module if and only if $X$ is.
\end{enumerate}
\end{Lemma}

\begin{proof}
(i) By \cite[p.341 line $-9$]{Thevenaz}, $k$ is a source algebra of $B$. 
Hence the assertion follows from Lemma {\bf\ref{TensorMorita}}(i).

(ii)--(iii) These follow from \cite[Proposition 1.2]{Kuelshammer1993}.
\end{proof}

Finally, we collect a few facts about Green correspondence, its
compatability with Brauer correspondence, and its transitivity
(see \cite[Chap.4, \S4]{NagaoTsushima}, for example).

\begin{Lemma}\label{GreenCorrespondence}
 \label{BrauerGreen}
Let $P$ be a $p$-subgroup of a finite group $G$, and
let $N$ and $H$ be subgroups of $G$ with
$N_G(P) \leqslant N \leqslant H \leqslant G$.
Furthermore, assume that 
$f$, $f_1$ and $f_2$ are the Green correspondences
with respect to $(G, P, H)$, $(H, P, N)$ and $(G,P,N)$,
respectively.
Then from the definition and properties of Green correspondence and
the Krull-Schmidt Theorem we get the following:
\begin{enumerate}
  \renewcommand{\labelenumi}{\rm{(\roman{enumi})}}
   \item We have $\mathfrak A(G,P,N) \subseteq \mathfrak A(G,P,H) 
    \cap \mathfrak A(H,P,N)$, where $\mathfrak A(G,P,N)$ and the 
    others are defined as in \cite[Chap.4, \S4]{NagaoTsushima}.
   \item For any indecomposable $kG$-module $X$ 
    with vertex in $\mathfrak A(G,P,N)$, the isomorphism
    $f_1\Big(f(X)\Big) \cong f_2(X)$ holds.
   \item
    Let $N=N_G(P)$, and let $A$, $B$, and $A_N$ be block algebras of
    $kG$, $kH$, and $kN$, respectively, such that they are Brauer
    correspondents with respect to $P$. Then any indecomposable
    $kG$-module $X$ belonging to $A$ such that a vertex of $X$ is in
    $\mathfrak A(G,P,N)$ 
    has its Green correspondent $f(X)$ belong to $B$.
\end{enumerate}
\end{Lemma}

\begin{proof} (i) and (ii) are clear.

 (iii) It follows from Green's result \cite[Chap.5
 Corollary~3.11]{NagaoTsushima} and Brauer's first main theorem that
 $f_2 (X)$ belongs to $A_N$. The Green correspondent $f(X)$ has a
 vertex in $\mathfrak A(G,P,N)$, and hence in $\mathfrak A(H,P,N)$. By
 (ii), $f_2 = f_1 \circ f$. Hence $f_2(X)
 = f_1\circ f(X)$ lies in the Brauer correspondent of $A$ which is
 $A_N$. Therefore, by the above, the block of $f(X)$ corresponds to $A_N$,
 namely, it is $B$.
\end{proof}

\section{Non-principal $2$-blocks of $\mathfrak S_5$ and
         \sf{M}$_{12}$}\label{Non-principal2-blocks}

By the "gluing" theorem given in \textbf{\ref{gluing}}, 
we want to
obtain a stable equivalence of Morita type between 
the non-principal $2$-block of
$\mathsf{Co}_3$ with 
a defect group $P = C_2 \times C_2 \times C_2$
and its Brauer correspondent in the normaliser $N_{\mathsf{Co}_3}(P)$.
To this end,
we need to consider non-trivial subgroups 
of $P$ and establish Morita equivalences between unique blocks 
of the associated centralisers in $\mathsf{Co}_3$ and $N_{\mathsf{Co}_3}(P)$.
The objective of this section is to show the existence of 
various Morita equivalences which will be required to apply
\textbf{\ref{gluing}}. 
The relevance of the groups related to $\mathfrak S_5$ and
$\mathsf{M}_{12}$, respectively, will be revealed in
in {\bf{\ref{2-local-Co3}}} in the next section.

For the remainder of this paper, we let the characteristic $p$ of $k$ be 2.

\begin{Lemma}\label{S5}
Set $G = \mathfrak S_5$.
\begin{enumerate}  \renewcommand{\labelenumi}{\rm{(\roman{enumi})}}
 \item 
There exists a unique block algebra $A$ of $kG$
with defect one. In fact, a defect group $T$ of $A$ 
is generated by a transposition.
\item 
Set $H = N_G(T)$. Then
   $H = C_G(T) 
   \cong T \times \mathfrak S_3 \cong D_{12}$.
\item 
$A$ is a nilpotent block algebra, 
   $k(A) = 2$, 
   $\ell(A) = 1$, and we can write 
${\mathrm{Irr}}(A) = \{
   \chi_4, \chi'_4 \}$ and ${\mathrm{IBr}}(A) = 
   \{ 4_{kG} \}$, 
where the number $4$ denotes the degree (dimension).
\item 
The unique simple $kG$-module $4_{kG}$ is a trivial source
module.
\item 
Let $B$ be a block algebra of $kH$ such that 
$B$ is the Brauer correspondent of $A$. Then 
   $k(B) = 2$, $\ell(B) = 1$, and
   we can write ${\mathrm{Irr}}(B) 
  = \{ \theta_2, \theta'_2 \}$ and
   ${\mathrm{IBr}}(B) = \{ 2_{kH} \}$, 
   where the number $2$ again gives
   the degree (dimension).
\item 
Let $\mathfrak{f}$ be the Green correspondence with respect to 
   $(G \times G, \Delta T, G \times H)$, and set 
   $M =  \mathfrak{f}(A)$.
   Then ${_A}M_B = 
   1_A{\cdot}kG{\cdot}1_B$ and $M$ induces a Puig
   equivalence between $A$ and $B$.
 \end{enumerate}
\end{Lemma}

\begin{proof}
(i)--(iii) and (v) are immediate by \cite[p.2]{Atlas}, 
and \cite[$A_5.2$ (mod~2)]{ModularAtlas} or
\cite[$A_5.2$ (mod~2)]{ModularAtlasProject}.

(iv) It follows from \cite[p.2]{Atlas} that 
${1_H}{\uparrow^G} = 1_G + \chi_4 + \chi_5$, 
where $\chi_i \in {\mathrm{Irr}}(G)$ and 
$\chi_i(1) = i$ for $i = 4, 5$. Thus, by (ii), 
${1_H}{\uparrow^G}{\cdot}1_A = \chi_4$, and 
hence ${k_H}{\uparrow^G}{\cdot}1_A = 4_{kG}$. 

(vi) We first show that $1_A{\cdot}kG{\cdot}1_B$ induces a Morita
equivalence between $A$ and $B$. 
To this end let 
$1_{\widehat{A}}{\cdot}\mathcal OG{\cdot} 1_{\widehat{B}}$ 
be its lift to $\mathcal O$, 
which is projective both as a left $\mathcal O G$-module and as
a right $\mathcal O H$-module. 
Moreover, it follows from (iii), (v), and
\cite[p.2]{Atlas} that
\[
\chi_4{\downarrow}_H{\cdot}1_B = \theta_2,
\quad
\chi'_4{\downarrow}_H{\cdot}1_B = \theta'_2
\]
by interchanging $\theta_2$ and $\theta'_2$ if necessary. 
Therefore
\[
\chi_4 \otimes_{\mathcal KA} 
(1_{\widehat{A}}{\cdot}\mathcal KG{\cdot}
1_{\widehat{B}}) 
= \theta_2,
\qquad
\chi'_4 \otimes_{\mathcal KA} 
(1_{\widehat{A}}{\cdot}
\mathcal KG{\cdot}1_{\widehat{B}}) 
= \theta'_2.
\]
Hence by \cite[0.2 Th{\'e}or{\`e}me]{Broue1990}
, we get that 
$1_{\widehat{A}}{\cdot}\mathcal OG{\cdot}1_{\widehat{B}}$ 
induces a Morita equivalence between $\widehat{A}$ 
and $\widehat{B}$, 
and so does $1_A{\cdot}kG{\cdot}1_B$ between 
$A$ and $B$. 
As $1_A{\cdot}kG{\cdot}1_B$ 
is a trivial source $k[G \times H]$-module with
vertex $\Delta P$, we infer that this even is a Puig equivalence.

Finally, let $(T,e)$ be a maximal $A$-Brauer pair. Then we know
$N_G(T,e) = H$ by (ii). Hence {\bf\ref{Green-GxG}} implies that 
$M | {1_A{\cdot}kG{\cdot}1_B}$. 
But it follows from Morita's Theorem,
see \cite[Sect. 3D Theorem (3.54)]{CR} that 
$1_A{\cdot}kG{\cdot}1_B$
already is indecomposable as an $(A,B)$-bimodule, 
implying that 
$M = 1_A{\cdot}kG{\cdot}1_B$.
\end{proof}

\begin{Lemma}\label{abelian2xS5}
Let $R$ be any finite 
$2$-group. 
Consider a finite group
$G = R \times \mathfrak S_5$, and let
$T$ be as in {\bf{\ref{S5}}}. 
Set $Q = R \times T$
and $H = N_G(Q)$. Let $A$ be a unique
non-principal block algebra of $kG$ 
with defect group $Q$, and let 
$B$
be a block algebra of $kH$ such that $B$ is 
the Brauer correspondent of
$A$. Then we get the following:
\begin{enumerate}
 \renewcommand{\labelenumi}{\rm{(\roman{enumi})}}
 \item 
$H = C_G(Q) = Q \times \mathfrak S_3$.
  \item 
Let $\mathfrak{f}$ be the Green correspondence with respect to 
$(G \times G, \Delta Q, G \times H)$, and set 
$M = \mathfrak{f}(A)$. Then
  $M \cong 1_A{\cdot}kG{\cdot}1_B$, 
  and $M$ induces a Puig equivalence
  between $A$ and $B$.
\end{enumerate}
\end{Lemma}

\begin{proof}
This follows from {\bf\ref{S5}}(vi) and
{\bf\ref{TensorMorita}}.
\end{proof}

\medskip
We next turn to the Mathieu group $\mathsf{M}_{12}$.

\begin{Lemma}\label{M12}
Let $G = \mathsf{M}_{12}$.
\begin{enumerate}
 \renewcommand{\labelenumi}{\rm{(\roman{enumi})}}
\item 
There exists a unique block algebra $A$ of $kG$ with 
  defect group $Q = C_2 \times C_2$.
 \item We can write 
  ${\mathrm{IBr}}(A) = \{ 16, 16^*, 144 \}$,
  where the numbers $16$ and $144$ denote dimensions (degrees).
  Moreover, all the simple $kG$-modules in $A$ are trivial source
  modules.
 \item 
Let $H = N_G(Q)$. Then 
  $H \cong {\mathfrak A}_4 \times \mathfrak S_3 
  \cong (Q \rtimes C_3) \times \mathfrak S_3$.
 \item 
Let $B$ be a block algebra of $kH$ such that 
  $B$ is the Brauer
  correspondent of $A$. Let $\mathfrak{f}$ 
  be the Green correspondence
  with respect to $(G \times G, \Delta Q, G \times H)$, and set 
$M = \mathfrak{f}(A)$. Then $M$ induces a Puig equivalence between 
  $A$ and $B$.
\end{enumerate}
\end{Lemma}

\begin{proof}
(i)--(iii) except the last part of (ii) are easy by \cite[p.33]{Atlas},
and \cite[$\mathsf{M}_{12}$ (mod 2)]{ModularAtlas} or 
\cite[$\mathsf{M}_{12}$ (mod 2)]{ModularAtlasProject}. 
Actually, using the character table of
$G$, it turns out that the conjugacy class {\rm{3B}} of $G$ is
a defect class of $A$. 
Hence $Q$ is a Sylow $2$-subgroup of the
centraliser $C_G(3B)={\mathfrak A}_4 \times C_3$, while the normaliser
$N_G(3B)={\mathfrak A}_4 \times \mathfrak S_3$ is a maximal subgroup of
$G$, containing $Q$ as normal subgroup.

It remains to show the last statement in (ii).
By \cite[p.33]{Atlas},
$G$ has a maximal subgroup $L \cong \mathrm{PSL}_2(11)$. Then again
\cite[p.33]{Atlas} yields that 
$1_L{\uparrow}^G{\cdot}1_A = \chi_{16} + \chi_{16}^*$, 
where $\chi_{16}(1) = \chi_{16}^*(1) = 16$. Set $X_{kG}
= k_L{\uparrow}^G{\cdot}1_A$. 
Then $X = 16 + 16^*$ as composition
factors. Since $\chi_{16} \:\, {\not=} \:\, \chi_{16}^*$, we get
by \cite[Chap.4 Theorem 8.9(i)]{NagaoTsushima} that $[X, X]^G = 2$. 
Therefore $X = 16 \oplus 16^*$. 
Hence $16$ and $16^*$ are both trivial source
$kG$-modules. Finally, we know that
$k_W{\uparrow}^G{\cdot}1_A = 144$, where $W$ 
is a maximal subgroup of
$G$ with $W = 2_{+}^{1+4}.\mathfrak S_3$. 
This shows that $144$ is also
a trivial source $kG$-module.

(iv) All elements of $Q - \{ 1 \}$ are conjugate in $H$, hence the
character table of $G$ \cite[p.33]{Atlas} shows that they all belong to
the conjugacy class {\rm{2A}} of $G$. Take any element $t \in Q - \{ 1 \}$,
and set $R = \langle t \rangle$. Thus we have
\[
C_G(R) \cong R \times \mathfrak S_5 
\qquad {\text{and}} \qquad C_H(R) \cong Q
\times \mathfrak S_3 \cong R \times (C_2 \times \mathfrak S_3).
\]
The algebra $kC_G(R)$ has a unique block algebra $A_R$ 
with the defect
group $Q$ since $k\mathfrak S_5$ has a unique block algebra with defect
group $C_2$, and similarly $kC_H(R)$ has a unique block algebra
$B_R$
with the defect group $Q$ since $k\mathfrak S_3$ has a unique block
algebra of defect zero. Moreover, we know by {\bf\ref{abelian2xS5}} that
$\mathfrak f_R(A_R)$ induces a Morita equivalence between 
$A_R$ and $B_R$, where $\mathfrak f_R$ is the Green correspondence
with respect to 
$(C_G(R) \times C_G(R), \Delta Q, C_G(R) \times C_H(R))$.
Thus it follows from {\bf\ref{gluing}} that $M$ induces a stable
equivalence of Morita type between $A$ and $B$.

Now, let $f$ be the Green correspondence with respect to 
$(G, Q, H)$. 
Take any simple $kG$-module $S$ in $A$. It follows from (ii),
\cite[3.7.Corollary]{Knoerr}, and \cite[Lemma 2.2]{Okuyama1981} that
$f(S)$ is a simple $kH$-module. Hence from {\bf\ref{SourceVertex}}(v)
and {\bf\ref{StableIndecomposableMorita}}(i) we obtain that $S
\otimes_A M$ is a simple $kH$-module in $B$. 
We then finally know
that $M$ realises a Morita equivalence between $A$ and $B$ by
{\bf\ref{StableIndecomposableMorita}(ii)}.
\end{proof}

\begin{Lemma}\label{C2xM12}
Set $R = C_2$, 
and let $G = R \times \mathsf{M}_{12}$.
\begin{enumerate}
 \renewcommand{\labelenumi}{\rm{(\roman{enumi})}}
\item 
There exists a unique block algebra $A$ of $kG$ with 
  defect group $P = R \times C_2 \times C_2$.
 \item 
We can write ${\mathrm{IBr}}(A) = \{ 16, 16^*, 144 \}$,
  where the numbers $16$ and $144$ give the dimensions (degrees).
  Moreover, all the simple $kG$-modules $16$, $16^*$, $144$ in $A$ are
  trivial source modules.
 \item 
Let $H = N_G(P)$. Then 
  $H = R \times \mathfrak A_4 \times \mathfrak S_3 
  \cong (P \rtimes C_3)
  \times \mathfrak S_3$. 
Note that 
$P \rtimes C_3 \cong R \times (Q \rtimes C_3)$ and 
$Q \rtimes C_3 \cong \mathfrak A_4$,
where $Q = C_2 \times C_2$.
 \item 
Let $B$ be a block algebra of $kH$ such that 
  $B$ is the Brauer correspondent of 
$A$. Let $\mathfrak f$ be the 
  Green correspondence
  with respect to $(G \times G, \Delta P, G \times H)$. 
Then $\mathfrak f(A)$ induces a Puig equivalence between 
  $A$ and $B$.
\end{enumerate}
\end{Lemma}

\begin{proof}
This follows from {\bf\ref{M12}}(iv) and {\bf\ref{TensorMorita}}.
\end{proof}

\section{Obtaining stable equivalences}\label{co3}

In this section, by using the lemmas in \S\S 2--3
we shall obtain
a stable equivalence of Morita type 
between the principal $2$-block of
the smallest Ree group $R(3)$ 
and the non-principal $2$-block of $\mathsf{Co}_3$ 
with defect group $C_2 \times C_2 \times C_2$ 
under consideration.
The following hypothesis determines our standard setting which we fix here 
for future reference.

\begin{Hypothesis}\label{hyp:Co3}
Let $G$ be the sporadic group $\mathsf{Co}_3$, and let $A$ be the block 
algebra of $kG$ with defect group 
$P = C_2 \times C_2 \times C_2$, 
see \cite[$\mathsf{Co}_3$]{ModularAtlasProject}, 
\cite[p.1879]{Landrock1978} and 
\cite[p.494 \S 2]{SuleimanWilson}.
Set $N = N_G(P)$, and let
$A_N$ be the Brauer correspondent of $A$ in $kN$. Furthermore, let $(P,e)$
be a maximal $A$-Brauer pair in $G$.

Let $Q$ be a subgroup of $P$ isomorphic to $C_2 \times C_2$, 
and $R$ one which is cyclic of order 2. Let $e_Q$ and $f_Q$ be block
idempotents of the block algebras of $kC_G(Q)$ and
$kC_H(Q)$, respectively, such that 
$(Q, e_Q) \subseteq (P, e)$ and $(Q, f_Q) \subseteq (P, e)$, 
see \cite[\S 10 p.346]{Thevenaz}. Similarly
define $e_R$ and $f_R$ by replacing $Q$ with $R$.
We denote by $F_{21}$ the Frobenius group of order $21$,
namely, $F_{21} \cong C_7 \rtimes C_3$, which is 
a maximal subgroup of $\GL_3(2)$. Also, let
$R(3) \cong \SL_2(8) \rtimes C_3$ be the smallest
Ree group, see \cite[p.6]{Atlas}.
\end{Hypothesis}

We first 
collect information on the subgroups of $\mathsf{Co}_3$ to consider.

\begin{Lemma}\label{2-local-Co3}
 Assume 
{\bf\ref{hyp:Co3}}. Then the following hold:

\begin{enumerate}
  \renewcommand{\labelenumi}{\rm{(\roman{enumi})}}
    \item
$N \cong (P \rtimes F_{21}) \times \mathfrak S_3
   \cong \Big((P \rtimes C_7) \rtimes C_3 \Big) \times \mathfrak S_3$.
    \item 
There is a maximal subgroup $H$ of $G$ such that
$N \leqslant H \cong (\SL_2(8) \rtimes C_3) \times \mathfrak S_3$, 
and
   $P \rtimes C_7$ is isomorphic to a Borel subgroup of $\SL_2(8)$.
    \item
$C_G(P) = C_H(P) = C_N(P) \cong P \times \mathfrak S_3$. 
    \item
There exists a unique block algebra $\beta$ of $k\mathfrak S_3$
such that $\beta$ has defect zero,
$\beta \cong {\mathrm{Mat}}_2(k)$ as $k$-algebras,
and $e kC_G(P) \cong kP \otimes \beta$.
    \item
$N_G(P,e)=N$.
    \item The inertial quotient $N_G(P,e)/C_G(P)$ is isomorphic to
$F_{21}$.
    \item
All elements of $P - \{1\}$ are conjugate in $N$.
That is, any subgroup of $P$ of order $2$ is 
conjugate to $R$ in $N$.
    \item
$C_G(R) \cong R \times \mathsf{M}_{12}$ 
and 
$C_H(R) = C_N(R) \cong R \times \mathfrak A_4 \times \mathfrak S_3 
        \cong (P \rtimes C_3) \times \mathfrak S_3$.
    \item
All subgroups of $P$ of order $4$ are conjugate in $N$.
That is, any subgroup of $P$ of order $4$ is
conjugate to $Q$ in $N$.
    \item
$C_G(Q) \cong Q \times \mathfrak S_5$ 
and
$C_H(Q) = C_N(Q) = C_H(P) \cong P \times \mathfrak S_3$.
     \item
Let $B = B_0(kR(3)) \otimes \beta$,
see {\rm{(iv)}} for $\beta$. 
Then $B$ is a block algebra of $kH$
with the defect group $P$, the block 
$B$ is the Brauer correspondent of $A$
and of $A_N$ in $H$, 
and we furthermore know that
$B$ and $B_0(kR(3))$ are Puig equivalent.
\end{enumerate}
\end{Lemma}

\begin{proof}
This is verified easily using {\sf{GAP}} \cite{GAP}, 
with the help of the smallest faithful permutation
representation of $G$ on $276$ points, 
available in \cite{AtlasRepresentation} in terms of so-called
standard generators \cite{Wilson}.
Since in \cite{AtlasRepresentation} also representatives
of the conjugacy classes of elements, as well as of the maximal subgroups
of $G$ are provided, all above-mentioned subgroups of $G$ 
can be constructed explicitly.

To begin with, using the character table of $G$ \cite[p.135]{Atlas},
it turns out that the conjugacy class {\rm{3C}} of $G$ 
is a defect class of $A$.
Hence $P$ is a Sylow $2$-subgroup of the centraliser $C_G(3C)$,
where by \cite[p.135]{Atlas} again we have 
$C_G(3C)\cong(\SL_2(8) \rtimes C_3) \times C_3$, while the normaliser
$H = N_G(3C)\cong (\SL_2(8) \rtimes C_3) \times \mathfrak S_3$
is a maximal subgroup of $G$. 

Using the data on subgroup fusions available
in \cite{CTblLib}, it follows that the elements of 
$P - \{1\}$ belong to the $2B$ conjugacy class of $G$,
hence \cite[p.134]{Atlas} shows that $C_G(R)\cong R \times \mathsf{M}_{12}$, which
is another maximal subgroup of $G$.
Moreover, it follows that 
$C_G(Q)\cong C_2 \times C_{\mathsf{M}_{12}}(2A)
       \cong C_2 \times (C_2 \times \mathfrak S_5)$,
where by \cite[p.33]{Atlas}
$C_2 \times \mathfrak S_5$ is a maximal subgroup of $\mathsf{M}_{12}$.
Finally, the structure of $C_H(P)$, $C_H(R)$, and $C_H(Q)$
follows from a consideration
of the action of $F_{21}\leqslant \GL_3(2)$ on the defect group $P$.

(xi) This follows by {\bf\ref{TensorPuigEquivalence}}.
\end{proof}

\begin{Notation}\label{H}
We use the notation $H$, $\beta$ and $B$ 
as in {\bf\ref{2-local-Co3}}(ii), (iv) and (xi),
respectively.
We denote the unique simple $k\mathfrak S_3$-module in $\beta$
by $2_{\mathfrak S_3}$.
\end{Notation}

It is now time to harvest what we have sown in our analysis of the 2-local
structure of $G$. In \textbf{\ref{stableEquivalence-A-AN}},
we use our previous results to obtain a stable
equivalence of Morita type between the blocks $A$ and $A_N$ via
\textbf{\ref{gluing}}. Similarly in \textbf{\ref{R(3)xS3}},
we derive a stable equivalence between the
blocks $B$ and $A_N$, which together with the first yields the stable
equivalence sought between $A$ and $B$ in 
\textbf{\ref{stableEquivalence-A-B}}.

\begin{Lemma}\label{R(3)xS3}
Let $\mathfrak f_1$ be the Green correspondence with respect to 
$(H \times H, \Delta P, H \times N)$,
and set $\mathfrak N = \mathfrak f_1(B)$. 
Then
$\mathfrak N$ induces a stable equivalence of Morita type 
between $B$ and $A_N$.
\end{Lemma}

\begin{proof}
By {\bf\ref{Green-GxG}}, $\mathfrak N | {1_B{\cdot}kH{\cdot}1_{A_N}}$.
We know by {\bf\ref{2-local-Co3}}(viii) and
{\bf\ref{2-local-Co3}}(x) that
\[
C_H(Q) = C_N(Q) = P \times \mathfrak S_3
\quad {\text{and}}  \quad
C_H(R)= C_N(R) = (P \rtimes C_3 ) \times \mathfrak S_3.
\]
Let $\mathbb A_Q$, $\mathbb A_R$, $\mathbb B_Q$ and $\mathbb B_R$
be the block algebras of
$kC_H(Q)$, $kC_H(R)$, $kC_N(Q)$ and $kC_N(R)$, respectively, 
such that they have $P$ as a defect group. Then
\[
 \mathbb A_Q = \mathbb B_Q 
             = kP \otimes k{\mathfrak S_3}{\cdot}\beta
             \cong {\mathrm{Mat}}_2(kP)
\quad {\text{and}}  \quad
 \mathbb A_R = \mathbb B_R
             = k[P\rtimes C_3] \otimes k{\mathfrak S_3}{\cdot}\beta,
\]
where the isomorphism is of $k$-algebras. Thus we obviously know that
\[
   \mathfrak f_Q(\mathbb A_Q) = \mathbb A_Q 
\quad {\text{and}}  \quad
   \mathfrak f_R(\mathbb A_R) = \mathbb A_R,
\]
where $\mathfrak f_Q$ and $\mathfrak f_R$ are
the Green correspondences with respect to 
\[
(C_H(Q) \times C_H(Q), \ \Delta P, \ C_H(Q) \times C_N(Q))
\qquad 
{\text{and}}  
\quad
(C_H(R) \times C_H(R), \ \Delta P, \ C_H(R) \times C_N(R)),
\]
respectively.
Thus $\mathfrak f_Q(\mathbb A_Q)$ induces a Morita equivalence
between $\mathbb A_Q$ and $\mathbb B_Q$, and
$\mathfrak f_R(\mathbb A_R)$ induces a Morita equivalence
between $\mathbb A_R$ and $\mathbb B_R$.
Therefore we get the assertion by {\bf\ref{gluing}}.
\end{proof}

\begin{Lemma}\label{stableEquivalence-A-AN}
Let $\mathfrak f_2$ be the Green correspondence
with respect to $(G \times G, \Delta P, G \times N)$,
and set $\mathfrak M = \mathfrak f_2(A)$.
Then we get
\begin{enumerate}
  \renewcommand{\labelenumi}{\rm{(\roman{enumi})}}
    \item
$\mathfrak M \mid 1_A{\cdot}kG{\cdot}1_{A_N}$. 
    \item
The bimodule $e_R \mathfrak M(\Delta R) f_R$ 
induces a Morita equivalence
between the block algebras $kC_G(R)e_R$ and $kC_N(R)f_R$.
    \item
The bimodule $e_Q \mathfrak M(\Delta Q) f_Q$ 
induces a Morita equivalence
between the block algebras $kC_G(Q)e_Q$ and $kC_N(Q)f_Q$.
    \item
$\mathfrak M$ induces a stable equivalence of Morita type
between $A$ and $A_N$.
\end{enumerate}
\end{Lemma}

\begin{proof}
(i)
This follows from {\bf\ref{2-local-Co3}}(v) and {\bf\ref{Green-GxG}}.

(ii)
Let $\mathfrak f_R$ be the Green correspondence
with respect to 
$(C_G(R) \times C_G(R), \Delta P, C_G(R) \times C_N(R))$.
We get from (i) and {\bf\ref{KoshitaniLinckelmann}} that
$\mathfrak f_R \Big(e_R kC_G(R) \Big)
 = e_R \mathfrak M(\Delta R) f_R$.
Hence we obtain the assertion by {\bf\ref{C2xM12}}.

(iii) 
Analogous to the proof of (ii) if we use {\bf\ref{abelian2xS5}}
instead of {\bf\ref{C2xM12}}.

(iv)
This follows by
{\bf\ref{C2xM12}} and {\bf\ref{abelian2xS5}},
(i)--(iii) and {\bf\ref{gluing}}.
\end{proof}

\begin{Lemma}\label{stableEquivalence-A-B}
There is an $(A, B)$-bimodule $M$ which satisfies
the following:
\begin{enumerate}
 \renewcommand{\labelenumi}{\rm{(\arabic{enumi})}}
   \item
${_A}M_B$ is indecomposable,
   \item
$({_A}M_B, {_B}{M^\vee}{_A})$ induces a stable equivalence 
of Morita type
between $A$ and $B$,
   \item
${_A}M_B \mid k_{\Delta P}\uparrow^{G \times H}$ and
${_B}{M^\vee}_A \mid k_{\Delta P}\uparrow^{H \times G}$,
   \item
the stable equivalence of Morita type induced by
${_A}M_B$ preserves vertices and sources,
   \item
for any indecomposable 
$X \in {\mathrm{mod}}{\text{-}}A$ with vertex 
in $\mathfrak A(G,P,N)$, it holds 
$(X\otimes_A M)_B = f(X) \oplus ({\mathrm{proj}})$,
where $f$ is the Green correspondence with respect to
$(G, P, H)$ 
{\rm{(}}recall that
$\mathfrak A(G,P,N) \subseteq 
\mathfrak A(G,P,H) \cap \mathfrak A(H,P,N)$
by {\bf\ref{GreenCorrespondence}}{\rm{\bf(i)}}{\rm{)}}.
\end{enumerate}
\end{Lemma}

\begin{proof}
Let $\mathfrak f_2$ be the Green correspondence with respect to
$(G \times G, \Delta P, G \times N)$, and set
$\mathfrak M = \mathfrak f_2(A)$. Let
$f_2$ be the Green correspondence with respect to $(G, P, N)$. 
Moreover, let $\mathfrak f_1$ be the Green correspondence with respect to
$(H \times H, \Delta P, H \times N)$, and set
$\mathfrak N = \mathfrak f_1(B)$.
Let $f_1$ be the Green correspondence with respect to $(H, P, N)$. 
Then by {\bf\ref{R(3)xS3}} and {\bf\ref{stableEquivalence-A-AN}} the
bimodules $\mathfrak{N}$ and $\mathfrak{M}$ induce stable equivalences, so
by {\bf\ref{SourceVertex}}(ii), and {\bf\ref{twoStableEquivalences}} 
there is a bimodule ${_A}M_B$ such that
\begin{equation*}\label{eq:stableEquivalence-A-B}\tag{$*$}
  {_A}(\mathfrak M \otimes_{A_N}\mathfrak N^{\vee})_B
=
  {_A}M_B \oplus ({\mathrm{proj}} \ (A,B){\text{-}}{\mathrm{bimodule}})
 \end{equation*}
and (1)--(4) hold.

It remains to show (5). 
Take any indecomposable $X \in {\mathrm{mod}}{\text{-}}A$ 
with a vertex which is in $\mathfrak A(G,P,N)$. 
Then it follows from \eqref{eq:stableEquivalence-A-B} that
\[
X \otimes_A(\mathfrak M \otimes_{A_N}\mathfrak N^{\vee})
\ = \
X \otimes_A(M \oplus ({\mathrm{proj}} \ 
(A,B){\text{-}}{\mathrm{bimodule}})).
\]
On the other hand, by {\bf\ref{GreenCorrespondence}}{(ii)}
we get
\begin{align*}
(X \otimes_A \mathfrak M)\otimes_{A_N}\mathfrak N^{\vee}
&=
[f_2 (X) \oplus ({\mathrm{proj}} \ A_N{\text{-}}{\mathrm{module}})]
       \otimes_{A_N} \mathfrak N^{\vee}        
\\
&=
( f_2 (X) \otimes_{A_N}\mathfrak N^{\vee})_B
\oplus 
(({\mathrm{proj}} \ A_N{\text{-}}{\mathrm{module}})
    \otimes_{A_N} \mathfrak N^{\vee})_B
\\
&= 
( f_2 (X) \otimes_{A_N}\mathfrak N^{\vee})_B
\oplus ({\mathrm{proj}} \ B{\text{-}}{\mathrm{module}})
\\  
&=
({f_1}^{-1}(f_2 (X)))_B \oplus ({\mathrm{proj}} \ B{\text{-}}{\mathrm{module}})
\\
&= f(X) \oplus ({\mathrm{proj}} \ B{\text{-}}{\mathrm{module}})
\end{align*}
\end{proof}

\section{Modules in $A$, $B$ and $A_N$}\label{blocks}

In the previous section, we have shown that there is a stable
equivalence of Morita type between the blocks $A$ and $B$. As
outlined in the introduction, our aim now is to verify that this
equivalence is in fact a Morita equivalence with the help of
\textbf{\ref{StableIndecomposableMorita}}. In other words, we need to
show that the associated tensor functor takes simple modules to simple
modules. Therefore in this intermediate section we collect all the
necessary information on the simple modules 
and some indecomposable modules lying in the three blocks we consider.

In addition to the notation of our standard hypothesis
{\bf\ref{hyp:Co3}}, we fix the following:

\begin{Lemma}[Suleiman-Wilson \cite{SuleimanWilson}]
\label{2-decompositionCo3}
The $2$-decomposition matrix
of $A$ is given in Table~{\rm{\ref{tab:Co3mod2}}},
where $S_1, \cdots , S_5$ are non-isomorphic simple $kG$-modules
in $A$ whose degrees are
$73600$, $896$, $896$, $19712$, $131584$, respectively. The two simple
modules $S_2$ and $S_3$ are dual to each other, while the remaining are
self-dual. There are two pairs $(\chi_6,\chi_7)$ and
$(\chi_{18},\chi_{19})$ of complex conjugate characters. All 
other $\chi$'s are real-valued.
\end{Lemma}

\begin{table}[ht]
 \centering
 \begin{tabular}{r|c|ccccc}
 {\rm{degree}} & \cite[p.135]{Atlas} & 
 $S_1$ & $S_2$ & $S_3 = S_2^*$ & $S_4$ & $S_5$  \\
 \hline
       $73\,600$ & $\chi_{29}$ & $1$  &  $.$    & $.$  &  $.$  &  $.$   \\
         $896$ & $\chi_{6}$  & $.$  &  $1$    & $.$  &  $.$  &  $.$   \\
         $896$ & $\chi_7 = \chi_{6}^*$  & $.$  &  $.$    & $1$  &  $.$  &  $.$   \\
       $93\,312$ & $\chi_{32}$ & $1$  &  $.$    & $.$  &  $1$  &  $.$   \\
       $20\,608$ & $\chi_{18}$ & $.$  &  $1$    & $.$  &  $1$  &  $.$   \\
       $20\,608$ & $\chi_{19}  = \chi_{18}^*$ & $.$  &  $.$    & $1$  &  $1$  &  $.$   \\
      $226\,688$ & $\chi_{38}$ & $1$  &  $1$    & $1$  &  $1$  &  $1$   \\
      $246\,400$ & $\chi_{39}$ & $1$  &  $1$    & $1$  &  $2$  &  $1$   \\
 \end{tabular}

\vspace{0.3cm}

 \caption{The $2$-modular decomposition matrix of $\mathsf{Co}_3$.}
 \label{tab:Co3mod2}
\end{table}

\begin{proof}
See \cite[\S 6]{SuleimanWilson}.
\end{proof}

\begin{Remark}
The $2$-blocks of ${\sf Co}_3$ have been studied before
by several other people,
see \cite[p.193 Table 6]{Fendel},
\cite[\S 7 p.1879]{Landrock1978}
and \cite[Theorems 3.10 and 3.11]{Landrock1981}.
\end{Remark}

\begin{Notation}\label{chi-Si}
We use the notation 
$\chi_{29}, \chi_6, \chi_7, \chi_{32}, 
 \chi_{18}, \chi_{19}, \chi_{38}, \chi_{39}$,
and $S_1, \cdots , S_5$ as in
\linebreak
{\bf\ref{2-decompositionCo3}}.
\end{Notation}

\begin{Lemma}\label{Knoerr}
All simple kG-modules $S_1, \cdots , S_5$ in $A$ have
$P$ as a vertex.
\end{Lemma}

\begin{proof}
See \cite[3.7.Corollary]{Knoerr}.
\end{proof}

\begin{Lemma}\label{simples-A-N}
We get the following:
\begin{enumerate}
  \renewcommand{\labelenumi}{\rm{(\roman{enumi})}}
    \item
$A_N = k[P \rtimes F_{21}] \otimes \beta 
     \cong {\mathrm{Mat}}_2 ( k[P \rtimes F_{21}])$,
as $k$-algebras.
    \item
We can write
${\mathrm{Irr}}(F_{21}) = \{ k, 1, 1^*, 3, 3^* \}$.
    \item
We can write
\begin{equation*}
 \begin{split}
\mathrm{IBr}(A_N) = 
\{ &\widetilde 2_0 = k_{P \rtimes F{_{21}}}  \otimes 2_{\mathfrak S_3}, \:
   \widetilde 2 = 1 \otimes 2_{\mathfrak S_3}, \\
   &\widetilde 2^* = 1^* \otimes 2_{\mathfrak S_3}, \: 
   \widetilde 6 = 3 \otimes 2_{\mathfrak S_3}, \:
   \widetilde 6^* = 3^* \otimes 2_{\mathfrak S_3}
\}.
 \end{split}
\end{equation*}
Note that there exists a unique simple $\widetilde 2_0$
which is self-dual. 
    \item
The trivial source $A_N$-modules with vertex $P$ are
precisely the simple $A_N$-modules. 
\end{enumerate}
\end{Lemma}

\begin{proof}
(i)--(iii) are easy by {\bf\ref{2-local-Co3}}
and the definition of $A_N$.

(iv) This follows from (iii) and the Green correspondence
\cite[Chap.4 Problem 10]{NagaoTsushima}.
\end{proof}

\begin{Lemma}\label{BrauerChar-B}
Set $\mathfrak R = R(3) \cong \SL_2(8) \rtimes C_3$.
We get the following:
\begin{enumerate}
  \renewcommand{\labelenumi}{\rm{(\roman{enumi})}}
 \item For the principal block of $k\mathfrak{R}$ we have
\[ {\mathrm{Irr}}(B_0(k\mathfrak R))=
\{ 1_{\mathfrak R}, \chi_1, \chi_1^*, 
   \chi_{7a}, \chi_{7b}, \chi_{7c}, \chi_{21}, \chi_{27} \},\]
and
\[{\mathrm{IBr}}(B_0(k\mathfrak R)) =  \{ k_{\mathfrak R}, 1, 1^*, 6, 12 \},\]
where the indices give the degrees (dimensions).
The simples $k_{\mathfrak R}, 6, 12$ are self-dual,
and the simples $k_{\mathfrak R}, 1, 1^*$ are
trivial source $k\mathfrak R$-modules.
   \item For the block $B$ we have
\[{\mathrm{Irr}}(B) = 
\{ \chi_{2a}, \chi_2, \chi_2^*, \chi_{14a}, \chi_{14b}, \chi_{14c},
   \chi_{42}, \chi_{54} \},\]
and
\begin{equation*}
\begin{split}
{\mathrm{IBr}}(B) = 
\{ &2_0 = k_{\mathfrak R} \otimes 2_{\mathfrak S_3}, \:
   2 = 1 \otimes 2_{\mathfrak S_3}, \\
   &2^* = 1^*  \otimes 2_{\mathfrak S_3}, \:
12 = 6 \otimes 2_{\mathfrak S_3}, \:
24 = 12\otimes 2_{\mathfrak S_3} \},
\end{split}
\end{equation*}
where the indices give the degrees (dimensions).
The simple $kH$-modules $2_0, 2, 2^*$ in $B$ are
trivial source modules,
the simple $kH$-modules $2_0$, $12$, $24$ are self-dual,
and all the simples in $B$ have $P$ as their vertices. 
\end{enumerate}
\end{Lemma}

\begin{proof}
(i) It follows from \cite[p.6]{Atlas}, and
\cite[$L_2(8).3$ (mod 2)]{ModularAtlas}
or \cite[$L_2(8).3$ (mod 2)]{ModularAtlasProject},
see {\bf\ref{2-local-Co3}}(xi).
Clearly, $k_{\mathfrak R}, 1, 1^*$ 
are trivial source $k\mathfrak R$-modules.

(ii) 
$2_{\mathfrak S_3}$ is a trivial source $k\mathfrak S_3$-module.
Therefore the simples $2_0, 2, 2^*$ are
trivial source $kH$-modules, by (i) and {\bf\ref{2-local-Co3}}(xi).
Finally, use \cite[3.7.Corollary]{Knoerr}.
\end{proof}

\begin{Notation}\label{simples-A_N-B}
We use the notation $\mathfrak R$, 
$\chi_{2a}, \chi_2, \chi_2^*, \chi_{14a}, \chi_{14b}, \chi_{14c},
  \chi_{42}, \chi_{54}$,
$\widetilde 2_0,  \widetilde 2, \widetilde 2^*,
 \widetilde 6, \widetilde 6^*$ and
$2_0, 2, 2^*, 12, 24$ 
as in
{\bf\ref{simples-A-N}} and {\bf\ref{BrauerChar-B}}.
\end{Notation}

\begin{Lemma}[Landrock-Michler \cite{LandrockMichler1980}]
\label{B-PIM}
The radical and socle series of projective indecomposable
$kH$-modules in $B$ are the following:
$$
\boxed{
\begin{matrix}
    2_0 \\
    12  \\
2_0 \ 2 \ 2^* \ 24 \\
    12 \ 12 \\
2_0 \ 2 \ 2^* \ 24 \\
    12  \\
    2_0 
\end{matrix}
},
\ \ \
\boxed{
\begin{matrix}
    2 \\
    12  \\
2_0 \ 2 \ 2^* \ 24 \\
    12 \ 12 \\
2_0 \ 2 \ 2^* \ 24 \\
    12  \\
    2 
\end{matrix}
}
\ \ \
\boxed{
\begin{matrix}
    2^* \\
    12  \\
2_0 \ 2 \ 2^* \ 24 \\
    12 \ 12 \\
2_0 \ 2 \ 2^* \ 24 \\
    12  \\
    2^* 
\end{matrix}
}
\ \ \ 
\boxed{
\begin{matrix}
    12  \\
2_0 \ 2 \ 2^* \ 24 \\
    12 \ 12 \ 12 \\
2_0 \ 2_0 \ 2 \ 2 \ 2^* \ 2^* \ 24 \ 24 \\
    12  \ 12 \ 12 \\
2_0 \ 2 \ 2^* \ 24 \\
    12 
\end{matrix}
}
\ \ \
\boxed{
\begin{matrix}
    24 \\
    12  \\
2_0 \ 2 \ 2^*\\
    12 \\
2_0 \ 2 \ 2^* \\
    12  \\
    24
\end{matrix}
}
$$
\end{Lemma}

\begin{proof}
This follows from
\cite[Theorem 3.9, Theorem 4.1]{LandrockMichler1980} and
{\bf\ref{BrauerChar-B}}.
\end{proof}

\begin{Lemma}\label{tsmR(3)C2}
Recall that $R$ is a subgroup of $P$ with $R \cong C_2$,
see {\bf\ref{hyp:Co3}}.
\begin{enumerate}
\renewcommand{\labelenumi}{\rm{(\roman{enumi})}}
    \item
The Scott module ${\mathrm{Scott}}(\mathfrak R, R)$ has the
radical and socle series
$$
\boxed
{
\begin{matrix}
k \\ 6 \\ 1 \ 1^* \ 12 \\ 6 \\ k 
 \end{matrix}
}
\leftrightarrow
1_{\mathfrak R} + \chi_{27}.
$$
    \item
A $kH$-module 
${\mathrm{Scott}}(\mathfrak R, R) \otimes 2_{\mathfrak S_3}$
has the radical and socle series
$$
\boxed
{
\begin{matrix}
2_0 \\ 12 \\ 2 \ 2^* \ 24 \\ 12 \\ 2_0 
 \end{matrix}
}
\leftrightarrow
\chi_{2a} + \chi_{54}.
$$
\end{enumerate}
\end{Lemma}

\begin{proof}
By {\bf\ref{BrauerChar-B}}(ii), it suffices to prove (i).
\cite[p.6]{Atlas} says that
$\mathfrak R$ has a maximal subgroup $M$ such that
$M = C_9 \rtimes C_6$, $|\mathfrak R: M| = 28$ and
$1_M{\uparrow}^{\mathfrak R} = 1_{\mathfrak R} + \chi_{27}$.
Set $X = k_M{\uparrow}^{\mathfrak R}$. Then
$X = 2 \times [k] + [1] + [1^*] + 2 \times [6] + [12]$, 
as composition factors
by \cite[$L_3(8).3$ (mod 2)]{ModularAtlas} and 
\cite[$L_3(8).3$ (mod 2)]{ModularAtlasProject}.
It holds by \cite[4 Thm.8.9(i)]{NagaoTsushima} that
$[X, X]^{\mathfrak R} = 2$,
$[X, k]^{\mathfrak R} = [k, X]^{\mathfrak R} = 1$ 
Thus, $X / \rad(X) \cong \soc(X) \cong k_{\mathfrak R}$. 
Now, it follows from \cite[Theorem 4.1]{LandrockMichler1980} that
$P(k_{\mathfrak R})$ has the following radical and socle series:
$$
P(k_{\mathfrak R}) \ = \ 
\boxed{
\begin{matrix}
    k   \\
    6 \\
 k \ 1 \ 1^* \ 12 \\
    6 \ 6 \\
 k \ 1 \ 1^* \ 12 \\
    6 \\
    k
\end{matrix}
}.
$$
Since there is an epimorphism 
$P(k_{\mathfrak R}) \twoheadrightarrow X$,
we infer $\soc(X)   \varsubsetneqq 
          \soc^2(X) \varsubsetneqq 
          \rad^2(X) \varsubsetneqq \rad(X)$ and
$\rad(X)/\rad^2(X)\cong \soc^2(X)/\soc(X)\cong 6$.
Thus $X$ has the radical and socle series as asserted.
By the definition of $X$, it holds that 
$X = {\mathrm{Scott}}(\mathfrak R, C_2)$,
see \cite[Chap.4 Theorem 8.4 and Corollary 8.5]{NagaoTsushima}.
\end{proof}

\begin{Lemma}\label{tsmR(3)C2C2}
Recall that $Q$ is a subgroup of $P$ with $Q \cong C_2 \times C_2$,
see {\bf\ref{hyp:Co3}}.
Set $U = {\mathrm{Scott}}(\mathfrak R, Q)$.
\begin{enumerate}
\renewcommand{\labelenumi}{\rm{(\roman{enumi})}}
    \item We have
$U \leftrightarrow 1_{\mathfrak R} + \chi_{7a} + 2 \times \chi_{27}$,
and
$ U= 4 \times [k_{\mathfrak R}] + 2 \times [1] + 2 \times [1^*] 
   + 5 \times [6] + 2 \times [12]$
as composition factors.
    \item
Set $V = U \otimes 2_{\mathfrak S_3}$. Then
$V$ is a trivial source $kH$-module in $B$ with vertex $Q$,
$V \leftrightarrow \chi_{2a} + \chi_{14a} + 2 \times \chi_{54}$,
and
$V = 4 \times [2_0] + 2 \times [2] + 2 \times [2^*] 
   + 5 \times [12] + 2 \times [24]$,
as composition factors.
\end{enumerate}
\end{Lemma}

\begin{proof}
(i) We know that $\mathfrak R$ has a subgroup $\mathfrak A_4$,
see \cite[p.6]{Atlas}. 
Clearly,
${\mathrm{Irr}}(\mathfrak A_4) = 
\{  1_{\mathfrak A_4}, \psi_1, \psi_2 = \psi_1^*, \psi_3 \}$
where $\psi_3$ has degree $3$.
It follows from 
computations with {\sf GAP} \cite{GAP} that
\begin{equation}\label{eq:1tsmR(3)C2C2}
1_{\mathfrak A_4}{\uparrow^{\mathfrak R}} \cdot 1_{B_0(k\mathfrak R)} 
= 
 1_{\mathfrak R} + \chi_{7a} + \chi_{21} + 3 \times \chi_{27},
\end{equation}
\begin{equation}\label{eq:2tsmR(3)C2C2}
\psi_1{\uparrow^{\mathfrak R}}\cdot 1_{B_0(k\mathfrak R)} 
= 
 \chi_1 + \chi_{7b} + \chi_{21} + 3 \times \chi_{27},
\end{equation}
\begin{equation}\label{eq:3tsmR(3)C2C2}
\psi_{1^* }{\uparrow^{\mathfrak R}}\cdot 1_{B_0(k\mathfrak R)} 
= 
 \chi_{1^*} + \chi_{7c} + \chi_{21} + 3 \times \chi_{27}.
\end{equation}
Let 
$X = k_{\mathfrak A_4}{\uparrow^\mathfrak R} \cdot 1_{B_0(k\mathfrak R)}$.
First, we want to claim that $P(12) \mid X$, where $P(12)$ 
is the projective cover $12$.

Set $S = \SL_2(8)$. By Clifford theory, we have
$12 \downarrow_S \, = \, 4_1 \oplus 4_2 \oplus 4_3$, 
where $4_1$, $4_2$, $4_3$
are non-isomorphic simple $kS$-modules in $B_0(kS)$ of dimension $4$,
see \cite[$L_2(8)$ (mod $2$)]{ModularAtlas}
and \cite[$L_2(8)$ (mod $2$)]{ModularAtlasProject}.
Let $V_1$ be the tautological $kS$-module, which is simple of dimension $2$,
and let $V_2$ and $V_3$ be its images under the action of the 
Frobenius automorphism of $\mathbb F_8$. Then 
the $V_i$ are pairwise non-isomorphic, and by \cite[p.220]{Alperin1979} 
we may assume that
\[
4_1 = V_1 \otimes V_2, \ \ 4_2 = V_2 \otimes V_3, \ \ 4_3 = V_3 \otimes V_1.
\] 

Set $g_a=\begin{pmatrix} 1 & a \\ 0 & 1 \end{pmatrix}\in S$
for all $a \in \mathbb F_8$. We may assume that 
$P = \{ g_a \mid  a \in \mathbb F_8 
     \} \leq S$,
namely, $P$ is a Sylow $2$-subgroup of $S$ with 
$P \cong C_2 \times C_2 \times C_2$,
and that 
$Q=\{ g_0, g_1, g_\alpha, g_{1+\alpha} \}$,
where $\alpha\in \mathbb F_8^\ast$ is a fixed primitive root,
hence $Q \cong C_2 \times C_2$.
Now the action of 
$g_0+ g_1+ g_\alpha+ g_{1+\alpha} 
=(1+g_1)(1+g_\alpha) \in kQ$
is easily described in terms of Kronecker products of matrices,
and it turns out that this element does not annihilate any of the
$kQ$-modules $4_i$. Therefore $4_i {\downarrow_Q}$ has a projective
indecomposable summand, and thus we infer that $4_i {\downarrow_Q}=P(k_Q)$.

We conclude
$12{\downarrow_Q} = 12{\downarrow_S}{\downarrow_Q}
 = (4_1 \oplus 4_2 \oplus 4_3){\downarrow_Q} \cong 3 \times P(k_Q)$,
 and it follows from \cite[Theorem 3]{Robinson1989} that
\begin{align*}
   3 
&= [P(k_Q) \mid
   12{\downarrow_Q}]^Q 
= [P(12) \mid {k_Q}{\uparrow^{\mathfrak R}}]^{\mathfrak R}
 = [P(12) \mid 
  k_Q{\uparrow^{\mathfrak A_4}}{\uparrow^{\mathfrak R}}]^{\mathfrak R}
\\
&= [P(12) \mid 
(k_{\mathfrak A_4} \oplus 1_{\mathfrak A_4} 
\oplus 1_{\mathfrak A_4}^*){\uparrow^{\mathfrak R}}]
^{\mathfrak R}
 = [P(12) \mid ({k_{\mathfrak A_4}}  {\uparrow^{\mathfrak R}} 
     \oplus {1_{\mathfrak A_4}}  {\uparrow^{\mathfrak R}} 
     \oplus {1_{\mathfrak A_4}^*}{\uparrow^{\mathfrak R}})]^{\mathfrak R}.
\end{align*}

Suppose that 
     $P(12) \nmid 
      k_{\mathfrak A_4}{\uparrow^{\mathfrak R}}$.
Then 
$3 \times P(12) \mid 
    ({1_{\mathfrak A_4}}{\uparrow^{\mathfrak R}} 
     \oplus {1_{\mathfrak A_4}^*}{\uparrow^{\mathfrak R}})$.
Since $P(12) \leftrightarrow \chi_{21} + \chi_{27}$ by 
\cite[$L_2(8).3$ (mod $2$)]{ModularAtlas} and
\cite[$L_2(8).3$ (mod 2)]{ModularAtlasProject}, we know by
\eqref{eq:2tsmR(3)C2C2} and \eqref{eq:3tsmR(3)C2C2} that
$3 \times \chi_{21} + 3 \times \chi_{27}$ is contained in
$(\chi_1 + \chi_{7b} + \chi_{21} + 3 \times \chi_{27}) + 
 (\chi_1^* + \chi_{7c} + \chi_{21} + 3 \times \chi_{27})$, which
contradicts the multiplicity of $\chi_{21}$.

Therefore $P(12) \mid k_{\mathfrak A_4}{\uparrow^{\mathfrak R}}$.
Since $P(12) \leftrightarrow \chi_{21} + \chi_{27}$ as seen above, 
it follows from \eqref{eq:1tsmR(3)C2C2} that
\[
    k_{\mathfrak A_4}{\uparrow^{\mathfrak R}}
    {\cdot} 1_{B_0(k\mathfrak R )}
    \ = \ X \oplus P(12)
\]
for a $k\mathfrak R$-module $X$ such that
\[
   X \leftrightarrow 1_{\mathfrak R} + \chi_{7a} + 2 \times \chi_{27}.
\]
Now, let $U = {\mathrm{Scott}}(\mathfrak R, Q)$, and 
hence ${U}{\mid}{X}$
since $Q$ is a Sylow $2$-subgroup of $\mathfrak A_4$, see
\cite[Chap.4 Corollary~8.5]{NagaoTsushima}.
By the definition of Scott modules and 
\cite[4 Thm.8.9(i)]{NagaoTsushima}, we know
$(\chi_{\widehat U}, 1_{\mathfrak R})^\mathfrak R = 1$.
Clearly, $\chi_{\widehat U} \, {\not=} \, 1_{\mathfrak R}$ 
since $Q \lneqq P$.
Since $P$ is a Sylow $2$-subgroup of $\mathfrak R$, it follows from
\cite[Chap.4, Theorem~7.5]{NagaoTsushima} that
$\dim_k(U)$ is even. This means that
$\chi_{\widehat U} \, {\not=} \, 1_{\mathfrak R} + 2 \times \chi_{27}$ and that
$\chi_{\widehat U} \, {\not=} \, 1_{\mathfrak R} + \chi_{7a} + \chi_{27}$.
If $\chi_{\widehat U} = 1_{\mathfrak R} + \chi_{7a}$ then
$\chi_{\widehat U}(2A) = 1 + (-1)= 0$ 
by \cite[p.6]{Atlas}, contradicting 
\cite[II Lemma~12.6]{Landrock1983} since $2A \in Q$.
Suppose that $\chi_{\widehat U} = 1_{\mathfrak R} + \chi_{27}$. Then since
$U$ is a trivial source $k\mathfrak R$-module, we get that $U$ has the same
radical and socle series of ${\mathrm{Scott}}(\mathfrak R, R)$ 
just by the same method as in {\bf\ref{tsmR(3)C2}}.
Since 
$ [U, {\mathrm{Scott}}(\mathfrak R, R)]^{\mathfrak R} = 2$
by \cite[4 Thm.8.9(i)]{NagaoTsushima}, we have
$U \cong {\mathrm{Scott}}(\mathfrak R, R)$, and hence $Q \cong R$
by \cite[Chap.4, Corollary~8.5]{NagaoTsushima}, again a contradiction.

Therefore we know that 
$\chi_{\widehat U} = 
  1_{\mathfrak R} + \chi_{7a} + 2 \times \chi_{27}$
and $U = X$, so that 
$ U= 4 \times [k_{\mathfrak R}] + 2 \times [1] + 2 \times [1^*] 
   + 5 \times [6] + 2 \times [12]$,
as composition factors.

(ii) This follows from (i) and {\bf\ref{2-local-Co3}}(xi).
\end{proof}

\begin{Remark}
We will not need the precise structure of 
$U = {\mathrm{Scott}}(\mathfrak R, Q)$. Still
we would like to remark that using the table of marks library of 
{\sf GAP} \cite{GAP}, and the facilities available
in the  {\sf MeatAxe} \cite{MA} and its extensions,
$U$ can actually be constructed
and analysed explicitly. In particular, it turns out that $U$
has Loewy length $5$, but its radical and socle series do not coincide;
they are
\[ \boxed{
\begin{matrix}
 k \ 6 \\ 
 k \ 1 \ 1^\ast \ 12 \\ 
 6 \ 6 \ 6 \\ 
 k \ k \ 1 \ 1^\ast \ 12 \\ 
 6 \\
\end{matrix}
}
\quad\quad\quad\text{and}\quad\quad\quad
\boxed{
\begin{matrix}
 6 \\ 
 k \ k \ 1 \ 1^\ast \ 12 \\ 
 6 \ 6 \ 6 \\ 
 k \ 1 \ 1^\ast \ 12 \\ 
 k \ 6 \\
\end{matrix}
},
\]
respectively.
\end{Remark}

\section{Images of simples in $A$ via Green correspondence}\label{img}
In this section we prove that the crucial hypothesis
of \textbf{\ref{StableIndecomposableMorita}} is fulfilled for 
the stable equivalence of Morita type we have
established in \textbf{\ref{stableEquivalence-A-B}}. Namely,
we show that simple modules in $A$ are taken to simple modules
in $B$. For the first four simples this is almost immediate, as this amounts
to determining the Green correspondents with respect to $(G,P,H)$, and these
are easily determined theoretically and computationally. The image of
the last simple $A$-module however, is more difficult to determine, and we
make use of our knowledge on the modules of the blocks $A$ and $B$ we have
gained in Section~\ref{blocks}.

\begin{Notation}\label{functor-F}
 We use the notation ${_A}M_B$, $f$, $f_1$ and $f_2$ as in
 {\bf\ref{stableEquivalence-A-B}}. Let $F: {\mathrm{mod}}{\text{-}}A
 \rightarrow {\mathrm{mod}}{\text{-}}B$ denote the functor giving the
 stable equivalence of Morita type of \textbf{\ref{stableEquivalence-A-B}},
 namely, in the notation of {\bf\ref{stableEquivalence-A-B}} we have $F(X) = X
 \otimes_A M$ for each $X \in {\mathrm{mod}}{\text{-}}A$. 
\end{Notation}

\begin{Lemma}\label{trivial-source-in-A}
The following hold:
\begin{enumerate}
  \renewcommand{\labelenumi}{\rm{(\roman{enumi})}}
    \item
$S_4 = 22 \otimes S_2$, where $22$ is a simple
$kG$-module in $B_0(kG)$.
    \item
We have
\[ 
{22}{\downarrow}_H = (6\otimes k_{\mathfrak S_3}) \oplus ({\mathrm{proj}}),
\quad 
{S_2}{\downarrow}_H
= 2 \oplus 110 \oplus ({\mathrm{proj}})
\quad {\mathrm{and}} \quad
(6\otimes k_{\mathfrak S_3}) \otimes 2 = 12,
\]
where $6\otimes k_{\mathfrak S_3}$ is a simple $kH$-module in 
$B_0(kH)=B_0(kR(3))\otimes B_0(k\mathfrak S_3)$, 
and $110$ is an 
indecomposable $kH$-module in $B_0(kH)$, hence
${S_2}{\downarrow}_H{\cdot}1_B = 2$ and
${S_2^*}{\downarrow}_H{\cdot}1_B = 2^*$.
    \item $12 \mid S_4{\downarrow}_H$. 
\end{enumerate}
\end{Lemma}

\begin{proof}
(i) This is obtained by 
\cite[p.502]{SuleimanWilson}, see
\cite[$\mathsf{Co}_3$ (mod 2)]{ModularAtlasProject},
and a direct computation with Brauer characters 
in {\sf GAP} \cite{GAP}.

(ii) By 
\cite[$L_3(8).3$ (mod 2)]{ModularAtlas} or
\cite[$L_3(8).3$ (mod 2)]{ModularAtlasProject},
except for the principal $2$-block $B_0(k[R(3)])$ of
$kR(3)=k[\SL_2(8)\rtimes C_3]$ there are only
three $2$-blocks of defect zero, consisting of the 
extensions of the Steinberg character of $\SL_2(8)$ to $R(3)$.
Hence it is easy to write down the block idempotents of $kR(3)$,
and similarly those of $k\mathfrak S_3$.
Thus, $H$ being a small group of order $9\,072$, 
using {\sf GAP} \cite{GAP} the block idempotents of $kH$ 
can be explicitly evaluated in a given representation.
This yields the block components, which are then
further analysed using the {\sf MeatAxe} \cite{MA}
and its extensions.

(iii)  
It follows from (i) and (ii) that
\begin{align*}
S_4{\downarrow}_H 
&= (22 \otimes S_2){\downarrow}_H 
 = 22{\downarrow}_H \otimes S_2{\downarrow}_H 
\\
&= \Big( (6\otimes k_{\mathfrak S_3}) \oplus ({\mathrm{proj}}) \Big) \otimes
       \Big(2 \oplus 110 \oplus ({\mathrm{proj}}) \Big)
\\
&= ((6\otimes k_{\mathfrak S_3}) \otimes 2) \oplus ({\mathrm{other}}) 
 = 12 \oplus ({\mathrm{other}}).
\end{align*}
\end{proof}

\begin{Lemma}\label{GreenCorrespondent-S2-S3-S4}
We have 
$ f(S_2) = 2$, $ f(S_2^*) = 2^*$,
$ f(S_4) = 12$, and hence that
$F(S_2) = 2$, $F(S_2^*) = 2^*$ and $F(S_4) = 12$.
\end{Lemma}

\begin{proof}
 By {\bf\ref{trivial-source-in-A}}(ii) the Green correspondents of
 $S_2$ and $S_2^*$ are immediate. By {\bf\ref{Knoerr}} all simple
 $A$-modules have vertex $P \in \mathfrak{A}(G,P,H)$, and 
 by \textbf{\ref{trivial-source-in-A}}(ii) the direct summands of
 $(6 \otimes k_{\mathfrak{S}_3}) \otimes 110$ lie in the principal block.
 Therefore by {\bf\ref{trivial-source-in-A}}(iii) and
 {\bf\ref{BrauerChar-B}}(ii)
 the simple module 
 $12$ is the unique summand of $S_4{\downarrow_H}$ in $B$ with vertex $P$.
 Hence $f(S_4)=12$.
 By {\bf\ref{stableEquivalence-A-B}}(5) and 
 {\bf\ref{StableIndecomposableMorita}}(i) the functor $F$ maps any simple
 $A$-module to its Green correspondent in $B$, and so the claim follows.
\end{proof}

\begin{Lemma}\label{trivial-source-S2-S3}
The simples $S_2$ and $S_2^*$ are trivial source
$kG$-modules with
$S_2 \leftrightarrow \chi_6$ and 
$S_2^* \leftrightarrow \chi_6^*$.
\end{Lemma}

\begin{proof}
We know by {\bf\ref{BrauerChar-B}}(ii) that
$2$ and $2^*$ are trivial source $kH$-modules.
Hence, by the definition of Green correspondence,
{\bf\ref{GreenCorrespondent-S2-S3-S4}}
and {\bf\ref{2-decompositionCo3}},
we get the assertion.
\end{proof}

\begin{Lemma}\label{trivial-source-S1}
The simple $kG$-module $S_1$ in $A$ is a trivial source module
with $S_1 \leftrightarrow \chi_{29}$.
\end{Lemma}

\begin{proof}
It follows from \cite[p.143]{Atlas} that $G$ has a maximal subgroup $L$
with $L = 2^{.}S_6(2)$. Then using {\sf{GAP}} \cite{GAP}, 
we know that ${1_L}{\uparrow}^G{\cdot}1_A = \chi_{29}$.
Hence the assertion follows by
{\bf\ref{2-decompositionCo3}}.
\end{proof}

\begin{Lemma}\label{GreenCorrespondent-S1}
We have $f(S_1) = 2_0$, and hence $F(S_1) = 2_0$.
\end{Lemma}

\begin{proof}
First, let $f'_1$ be the Green correspondence with respect to
$(R(3), P, P \rtimes F_{21})$. Clearly,
$f'_1(k_{R(3)}) = k_{P \rtimes F_{21}}$. 
Since $ f_1$ is the Green correspondence with respect to 
$(H, P, N)=
(R(3) \times \mathfrak S_3,P, (P \rtimes F_{21})\times \mathfrak S_3)$,
we know that
$ f_1(k_{R(3)} \otimes 2_{\mathfrak S_3})
    = k_{P \rtimes F_{21}}  \otimes 2_{\mathfrak S_3}$,
namely, $ f_1(2_0) = \widetilde 2_0$.

By {\bf\ref{GreenCorrespondence}}(ii), $ f_1 \circ  f = f_2$.
Thus it follows from {\bf\ref{Knoerr}},
{\bf\ref{trivial-source-S1}} and {\bf\ref{BrauerGreen}}(iii) that
$ f_1 \circ  f (S_1)$ is a
trivial source $kN$-module in $A_N$ with vertex $P$.
Hence {\bf\ref{simples-A-N}}(iv) implies that
\[
 f_1 \circ  f (S_1) \: \in \:
\{ \widetilde 2_0, \: \widetilde 2, \: 
\widetilde 2^*, \: \widetilde 6, \: \widetilde 6^* \}.
\]
Then since $S_1$ is self-dual by 
{\bf\ref{2-decompositionCo3}}, we know that
$ f_1 \circ  f (S_1)$ is also self-dual.
Therefore 
$ f_1 \circ  f (S_1) = \widetilde 2_0$,
giving $ f_1 \circ  f (S_1) =  f_1(2_0)$.
This implies that $f(S_1) = 2_0$.
Hence we get the assertion from
{\bf\ref{stableEquivalence-A-B}}(5)
and
{\bf\ref{StableIndecomposableMorita}}(i).
\end{proof}

\begin{Lemma}\label{Ext}
The following hold:
\begin{enumerate}
  \renewcommand{\labelenumi}{\rm{(\roman{enumi})}}
    \item
$\Ext_A^1(S_1, S_2) = \Ext_A^1(S_1, S_2^*) 
 = \Ext_A^1(S_2, S_1) = \Ext_A^1(S_2^*, S_1) = 0$.
\smallskip
    \item
$\Ext_A^1(S_2, S_2^*) = \Ext_A^1(S_2^*, S_2) = 0$.
\smallskip
    \item
$\dim_k[\Ext_A^1(S_1, S_4)] =\dim_k[\Ext_A^1(S_4, S_1)] = 1$.  
\end{enumerate}
\end{Lemma}

\begin{proof}
 By \textbf{\ref{GreenCorrespondent-S1}} and
 \textbf{\ref{GreenCorrespondent-S2-S3-S4}} we know the simple images of the
 simple modules given under the stable equivalence $F$ of
 \textbf{\ref{functor-F}}. Hence the results are immediate by looking at the
 $B$-PIMs in \textbf{\ref{B-PIM}},
see \cite[X.2 Proposition 1.12]{AuslanderReitenSmalo}
or \cite[\S 5]{CarlsonETH} for instance.
\end{proof}

\begin{Lemma}\label{TopOfF(S5)}
All composition factors of $F(S_5)/\rad(F(S_5))$ and $\soc(F(S_5))$ 
are isomorphic to the simple module $24$.
\end{Lemma}

\begin{proof}
Take any simple $kH$-module $T$ in $B$ such that
$T \, {\not\cong} \, 24$. Then we know by 
{\bf\ref{BrauerChar-B}}, {\bf\ref{GreenCorrespondent-S2-S3-S4}}
and {\bf\ref{GreenCorrespondent-S1}} that
$T = F(S_i)$ for $i \in \{ 1, 2, 3, 4 \}$, where
$S_3 = S_2^*$. It then follows from
\cite[II Lemma 2.7 and Corollary 2.8]{Landrock1983}
and {\bf\ref{functor-F}} that
$  \Hom_B(F(S_5), T) = {\underline{\Hom}}_B(F(S_5), T)
       = {\underline{\Hom}}_B(F(S_5), F(S_i))
       \cong {\underline{\Hom}}_A(S_5, S_i)
       = \Hom_A(S_5, S_i) = 0$.
Thus we get the assertion for the head of $F(S_5)$. The assertion for the
socle follows by the same argument and considering
$\Hom_B(T,F(S_5))$ instead.
\end{proof}
We can now finally prove that also the image of the last remaining simple
$A$--module $S_5$ under $F$ is a simple $B$-module.

\begin{Lemma}\label{F(S5)}
We have $F(S_5) = 24$.
\end{Lemma}

\begin{proof}
By \cite[p.134]{Atlas}, $G$ has a maximal subgroup
$\mathfrak U = {\mathrm{U}}_3(5) \rtimes \mathfrak S_3$.
Set $X = k_{\mathfrak U}{\uparrow^G}{\cdot}1_A$.
By calculations in {\sf GAP} \cite{GAP} we know that
$1_{\mathfrak U}{\uparrow^G}{\cdot}1_A = \chi_{29} + \chi_{39}$,
so that
\begin{equation}\label{eq:4F(S5)}
X \ \leftrightarrow \ \chi_{29} + \chi_{39}.
\end{equation}
Hence, 
by {\bf\ref{2-decompositionCo3}}
\begin{equation}\label{eq:5F(S5)}
X = 2 \times S_1 + S_2 + S_2^* + 2 \times S_4 + S_5, \  
\text{ as composition factors}.
\end{equation}
Since $S_1$, $S_2$ and $S_2^*$ are trivial source $kG$-modules
by {\bf\ref{trivial-source-S1}} and {\bf\ref{trivial-source-S2-S3}},
it follows from \eqref{eq:4F(S5)}, 
{\bf\ref{2-decompositionCo3}} and 
\cite[4 Thm.8.9(i)]{NagaoTsushima} that
$$
[S_1, X]^G = [X, S_1]^G = 1, \qquad
[S_2, X]^G = [X, S_2]^G = [S_2^*, X]^G = [X, S_2^*]^G = 0.
$$

If $[S_5,X]^G \not= 0$ or $[X,S_5]^G \not= 0$, then the
self-duality of $X$ and $S_5$ 
implies that $S_5 \mid X$, and hence $S_5$ is a trivial source
$kG$-module, so that $S_5$ is liftable to $\mathcal O$ by
\cite[4 Thm.8.9(iii)]{NagaoTsushima}, which contradicts to
{\bf\ref{2-decompositionCo3}}. Hence
\[
  [S_5,X]^G = [X, S_5]^G = 0.
\]
Assume $[S_4,X]^G \not= 0$ or $[X, S_4]^G \not= 0$. Then again
the self-dualities of $X$ and $S_4$ in {\bf\ref{2-decompositionCo3}}
say that both are non-zero. Thus we have endomorphisms
$\psi_1$, $\psi_2$ and $\psi_3$ of $X$ such that
$\psi_1 = id_X$, $\Im (\psi_2) \cong S_1$ and $\Im (\psi_3) \cong S_4$.
This means $[X,X]^G \geqslant 3$. 
But \cite[4 Thm.8.9(i)]{NagaoTsushima}
and \eqref{eq:4F(S5)} yield that
$[X,X]^G = 2$, a contradiction. Thus
$[S_4, X]^G = [X, S_4]^G = 0$.
These imply that 
\begin{equation}\label{eq:6F(S5)}
X/\rad(X) \cong \soc(X) \cong S_1.
\end{equation}
Hence $X$ is indecomposable.
Set $X_0 = \rad(X)/\soc(X)$, the heart of $X$. Thus \eqref{eq:5F(S5)} 
implies 
\begin{equation}\label{eq:7F(S5)}
X_0 = S_2 + S_2^* + 2 \times S_4 + S_5,
\text{ as composition factors}.
\end{equation}
By {\bf\ref{Ext}}(i), it holds
\[
  [X_0, S_2]^G = [X_0, S_2^*]^G 
= [S_2, X_0]^G = [S_2^*, X_0]^G = 0.
\]
Moreover, {\bf\ref{Ext}}(iii) yields that
$X_0/\rad(X_0) \mid (S_4 \oplus S_5)$.
These imply that the radical and socle series of $X$ is
one of the following:
\begin{equation}\label{eq:8F(S5)}
X \ = \ \ 
\boxed{
  \begin{matrix}
     S_1 \\
     S_4 \\
  S_2 \ S_2^* \ S_5 \\
     S_4 \\
     S_1 
  \end{matrix}
      },
\qquad\quad
\boxed{
  \begin{matrix}
     \ & S_1 & \ 
\\
    \boxed{
     \begin{matrix}
      S_4 \\
    S_2 \ S_2^* \\
      S_4 
     \end{matrix}
          }
      & \oplus & S_5 
\\
     \ & S_1 & \
   \end{matrix}
},
\qquad\quad
   \boxed{
      \begin{matrix}
       S_1  \ \\
       S_4  \ \\
       S_2  \ \\
       S_5  \ \\
       S_2^* \: \\
       S_4  \ \\
       S_1  \
      \end{matrix}
          }
\qquad
{\text{or}}
\qquad
   \boxed{
      \begin{matrix}
       S_1 \ \\
       S_4 \ \\
       S_2^* \:\\
       S_5 \ \\
       S_2 \ \\
       S_4 \ \\
       S_1  \
      \end{matrix}
          }.
\end{equation}

\qquad\qquad\qquad

Now, it follows from
{\bf\ref{functor-F}}, 
\cite[II Lemma 2.7 and Corollary 2.8]{Landrock1983},
{\bf\ref{GreenCorrespondent-S2-S3-S4}} and \eqref{eq:6F(S5)} that
\begin{align*}
\Hom_B(F(X),2) &= {\underline{\Hom}}_B(F(X),2)
                = {\underline{\Hom}}_B(F(X),F(S_2))
\\
&\cong {\underline{\Hom}}_A(X, S_2) = \Hom_A(X, S_2) = 0.
\end{align*}
Hence $[F(X),2)]^B = 0$. Similarly we obtain
$[F(X),2^*]^B = 0$ and $[F(X), 12]^B = 0$
and $[F(X), 2_0]^B = 1$.
Similar for $\soc(F(X))$, too. Thus, by 
{\bf\ref{BrauerChar-B}}, we know that
\begin{equation}
 F(X)/ \rad(F(X)) \cong 2_0 \oplus (r \times 24) \text{ and }
\soc(F(X)) \cong 2_0 \oplus (r' \times 24)
\end{equation}
for some $r, r' \geqslant 0$.
By {\bf\ref{functor-F}}, we have
\begin{equation}
F(X) = Y \oplus ({\text{proj}} \ B{\text{-module}})
\end{equation}
for a non-projective indecomposable $kH$-module $Y$ in $B$.
Thus, by {\bf\ref{GreenCorrespondent-S1}} and 
{\bf\ref{OmegaCommutesWithFStable}}(i)-(ii) 
we have
\begin{equation}\label{eq:11F(S5)}
 2_0 {\Big|} Y/\rad(Y) \ \ \text{ and } \ \ 2_0 {\Big|} \soc(Y).
\end{equation}
Recall that $2_0 = k_{\mathfrak R} \otimes 2_{\mathfrak S_3}$ in
{\bf\ref{BrauerChar-B}}(ii). Since $B$ and $B_0(k\mathfrak R)$ are
Puig equivalent by {\bf\ref{2-local-Co3}}(xi), 
and $Y$ is a trivial source module by
{\bf\ref{stableEquivalence-A-B}},
it follows that
$Y \cong {\mathrm{Scott}}(\mathfrak R, S) \otimes 2_{\mathfrak S_3}$
for a subgroup $S$ of $P$. 
Clearly $S \not= 1$ since $Y$ is non-projective indecomposable.
If $S = P$ then \eqref{eq:11F(S5)} yields $Y = 2_0$, so that
$F(X) = 2_0 \oplus ({\text{proj}})$ and $F(S_1) = 2_0$
by {\bf\ref{GreenCorrespondent-S1}}. This is a contradiction 
since $X$ is non-projective indecomposable
and non-simple. 
Thus $S \cong Q$ or $S \cong R$.

Suppose that $S \cong Q$, namely 
$Y \cong \mathrm{Scott}(\mathfrak R, Q) \otimes 2_{\mathfrak S_3}$.
Then it follows by {\bf\ref{tsmR(3)C2C2}}(ii) that
\[
Y \leftrightarrow \chi_{2a} + \chi_{14a} + 2 \times \chi_{54},
\]
and we have
\begin{equation}\label{eq:12F(S5)}
Y = 4 \times [2_0] + 2 \times [2] + 2 \times [2^*] 
     + 5 \times [12] + 2 \times [24],
\text{ as composition factors}.
\end{equation}
We know by {\bf\ref{GreenCorrespondent-S1}} 
and {\bf\ref{GreenCorrespondent-S2-S3-S4}} that
\[
F(S_1) = 2_0, \ F(S_4) = 12, \ F(S_2) = 2, \ F(S_2^*) = 2^*.
\]
Thus it follows by \eqref{eq:6F(S5)}, \eqref{eq:8F(S5)} and 
{\bf\ref{OmegaCommutesWithFStable}}(i)-(ii) that
we can {\sl strip off} \: 
$2 \times S_1$, $ 2 \times S_4$, $S_2$, and $S_2^*$
from the top of $X$ and from the bottom of $X$,
and also $2 \times [2_0]$, $ 2 \times [12]$, $[2]$, and $[2^*]$
from the top of $Y$ and from the bottom of $Y$
sequentially,
by looking at \eqref{eq:8F(S5)} and \eqref{eq:12F(S5)}. 
Consequently by {\bf\ref{StableIndecomposableMorita}}(i), 
we have $F(S_5) = Z$ 
for an indecomposable $kH$-module $Z$ in $B$ such that
$Z = 2 \times [2_0] + [2] + [2^*] + 3 \times [12] + 2 \times [24]$,
as composition factors.
Then 
{\bf\ref{TopOfF(S5)}} yields 
$Z/\rad(Z) \cong \soc(Z) \cong 24$ and
$\rad(Z)/\soc(Z) =  [2_0] + [2] + [2^*] + 3 \times [12]$
as composition factors, which contradicts {\bf\ref{B-PIM}}.

Therefore $S \cong R$ and 
$Y \cong {\mathrm{Scott}}(\mathfrak R, R) \otimes 2_{\mathfrak S_3}$.
Hence we get by {\bf\ref{tsmR(3)C2}}(ii) that
\begin{equation}\label{eq:13F(S5)}
F(X) \ = \ Y\oplus ({\text{proj}}), \qquad
Y \ = \  
\boxed{
  \begin{matrix}
      2_0 \\
      12  \\
   2 \ 2^* \ 24 \\
      12  \\
      2_0 
  \end{matrix}
      }.
\end{equation}
Thus by the same {\sl stripping-off} method
{\bf\ref{OmegaCommutesWithFStable}}(i)-(ii)
taken above, we can subsequently strip off
$2 \times S_1$, $ 2 \times S_4$, $S_2$, and $S_2^*$
from the top of $X$ and the bottom of $X$,
and also $2 \times [2_0]$, $2 \times [12]$, $[2]$, and $[2^*]$
from the top of $Y$ and the bottom of $Y$, 
by looking at \eqref{eq:8F(S5)} and \eqref{eq:13F(S5)}.
Hence we arrive at $F(S_5) = 24 \oplus ({\text{proj}})$,
so that {\bf\ref{StableIndecomposableMorita}} yields
$F(S_5) = 24$.
\end{proof}

\begin{Remark}
We know by {\bf\ref{corollary-R(q)}} that
the block $A$ of $G$ and the principal $2$-block
$B_0(kR(3))$ of $R(3)$ are Puig equivalent.
Let $X$ be the same as in the proof of {\bf\ref{F(S5)}}.
Thus it follows from
{\bf\ref{tsmR(3)C2}}(i)-(ii) and the proof of {\bf\ref{F(S5)}}
that the radical and socle series of $X$ is actually 
the first one in \eqref{eq:8F(S5)} in the proof of {\bf\ref{F(S5)}},
and that $X$ is a trivial source $kG$-module in $A$
with vertex $C_2$.
\end{Remark}

\section{Proof of the main results}\label{proof}

\noindent{\bf Proof of {\bf\ref{2ndMainTheorem}}.}
First of all, consider the blocks $A$ and $B$ over $k$, namely,
$A$ and $B$ are block algebras of $kG$ and $kH$, respectively.
Hence $M$ is a $(kG, kH)$-bimodule.
We know by {\bf\ref{stableEquivalence-A-B}}(ii) and
{\bf\ref{functor-F}} that
the functor $F$ defined by $M$ realises a stable
equivalence of Morita type between $A$ and $B$.
It follows from 
{\bf\ref{2-decompositionCo3}},
{\bf\ref{GreenCorrespondent-S2-S3-S4}},
{\bf\ref{GreenCorrespondent-S1}} and 
{\bf\ref{F(S5)}} that,
for any simple $kG$-module $S$ in $A$,
$F(S)$ is a simple $kH$-module in $B$.
Hence, {\bf\ref{StableIndecomposableMorita}}(ii) yields
that ${_A}{M}_B$  realises a Morita equivalence between
$A$ and $B$.
Since $M$ is a $\Delta P$-projective
trivial source $k[G \times H]$-module,
the Morita equivalence is a Puig equivalence by
\cite[Remark 7.5]{Puig1999} or
\cite[Theorem 4.1]{Linckelmann2001}
(note that this was independently observed by
L.~Scott).
Moreover, by 
\cite[4 Thm.8.9(i)]{NagaoTsushima},
the Morita equivalence
lifts from $k$ to $\mathcal O$;
see also \cite[(38.8)Proposition]{Thevenaz} or
\cite[7.8.Lemma]{Puig1988Inv}.
\hfill$\Box$

\bigskip
\noindent{\bf Proof of Corollary {\bf\ref{corollary-R(q)}}.}
This follows by 
{\bf\ref{2ndMainTheorem}}, {\bf\ref{LandrockMichlerOkuyama}}
and {\bf\ref{TensorPuigEquivalence}}.
\hfill$\Box$

\bigskip
\noindent{\bf Proof of Theorem {\bf\ref{MainTheorem}}.}
This follows from
{\bf\ref{corollary-R(q)}}, 
{\bf\ref{TensorPuigEquivalence}} and {\bf\ref{2-local-Co3}} (i).
\hfill$\Box$

\bigskip

\begin{appendix}\label{sec:app}
\section{Properties of the stable equivalences considered}
In this appendix we collect some fundamental properties of the stable
equivalences which are found throughout this paper, and in particular of the
stable equivalence $F$ of {\bf \ref{functor-F}}. For the large part, 
these properties are used at several steps in this paper, but they are also
of independent interest, as a referenceable collection with proofs
is desirable. Also, in this section, 
we aim to supply more general hypotheses for clarity.

The first fundamental property 
we collect is the 
the following ``stripping off''-method, which enables us to
reduce the problem of determining the image of a module under
a stable equivalence to determining the images of its head and socle
components; the proof of {\bf\ref{F(S5)}} bears testimony of the utility
of this lemma. See also
\cite{KoshitaniKunugiWaki2004} in which 
{\bf\ref{OmegaCommutesWithFStable}} is firstly conceived and
applied.

\begin{Lemma}\label{OmegaCommutesWithFStable}
Let $A$ and $B$ be finite dimensional $k$-algebras
for a field $k$ such that $A$ and $B$ are both self-injective.
Let $F$ be a covariant functor such that
 \begin{enumerate}
  \renewcommand{\labelenumi}{\rm{(\arabic{enumi})}}
   \item $F$ is exact.
   \item If $X$ is a projective $A$-module, then $F(X)$ is a 
    projective $B$-module,
   \item $F$ induces a stable equivalence from
    $\mod{\text{-}}A$ to $\mod{\text{-}}B$.
 \end{enumerate}
Then the following holds:
\begin{enumerate}
 \renewcommand{\labelenumi}{\rm{(\roman{enumi})}}
 \item
{\sf (Stripping-off method, case of socle)} \, 
  Let $X$ be a projective-free $A$-module, and write
  $F(X) = Y \oplus ({\mathrm{proj}})$ for a projective-free $B$-module $Y$. 
  Let $S$ be a simple $A$-submodule of $X$, and set $T = F(S)$. 
  Now, if $T$ is a simple $B$-module, then we may assume that $Y$
  contains $T$ and that 
  $$ F(X/S)= Y/T \oplus ({\mathrm{proj}}) .$$
 \item 
{\sf (Stripping-off method, case of radical)} \,
  Similarly, 
  let $X$ be a projective-free $A$-module, and write
  $F(X) = Y \oplus ({\mathrm{proj}})$ for a projective-free $B$-module $Y$. 
  Let $X'$ be an $A$-submodule of $X$ such that $X/X'$ is simple,
  and set $T = F(X/X')$. Now, if $T$ is a simple $B$-module,
  then we may assume that $T$ is an epimorphic image of $Y$ and that
  $$ {\mathrm{Ker}}(F(X) \twoheadrightarrow T) ={\mathrm{Ker}}(Y
  \twoheadrightarrow T) \oplus ({\mathrm{proj}}) .$$
\end{enumerate}
\end{Lemma}

\begin{proof}

(i)-(ii)
The assertions are got
from \cite[II Lemma 2.7 and Corollary 2.8]{Landrock1983}
 and \cite[1.11.Lemma]{KoshitaniKunugiWaki2004},
just as in
\cite[3.25.Lemma and 3.26.Lemma]{KoshitaniKunugiWaki2004}.
\end{proof}

Next, we want to show that the stable equivalence of Morita type also
commutes with taking the contragredient module if $A$ and $B$ are blocks
of group algebras. This is made precise in 
{\bf\ref{k-dual-functor}}(iv), 
but first we place ourselves into a more general context.

\begin{Lemma}\label{k-dual-functor}
Let $A$ and $B$ be finite dimensional $k$-algebras for a field $k$.
 \begin{enumerate}
 \renewcommand{\labelenumi}{\rm{(\roman{enumi})}}
   \item
Assume that $X \in {\mathrm{mod}}{\text{-}}A$, and
$M \in A{\text{-}}{\mathrm{mod}}{\text{-}}B$, and that
$_A M$ is projective. Then the correspondence
$$
\Phi :     {_B}(M^\vee \otimes_A X^\circledast) 
        \, \overset{\rightarrow}
        \, _B[ (X \otimes_A M)^\circledast]
$$
defined by
$$
\Big[ \Phi(\psi\otimes_A \theta)\Big](x \otimes_A m) 
  = \theta \Big(x {\cdot}\psi(m)\Big)
$$
for $\psi \in M^\vee$, $\theta \in X^\circledast$ and $m \in M$,
is an {isomorphism} of left $B$-modules.
\smallskip
   \item
Assume that $Y \in A{\text{-}}{\mathrm{mod}}$, and
$N \in B{\text{-}}{\mathrm{mod}}{\text{-}}A$, and that
$N_A$ is projective.
Then the correspondence
$$
\Theta :     (Y^\circledast \otimes_A N^\vee)_B 
        \, \overset{\rightarrow}
        \, [(N \otimes_A Y)^\circledast]_B
$$
defined by
$$
\Big[ \Theta(\theta\otimes_A \psi)\Big](n \otimes_A y) 
  = \theta \Big( \psi(n)\cdot y \Big)
$$
for $\psi \in N^\vee$, $\theta \in Y^\circledast$ and $n \in N$,
is an {isomorphism} of right $B$-modules.
\smallskip
    \item
     If $A$ moreover is a symmetric algebra, 
with symmetrising form $t \in \Hom_k(A,k)$, 
then as $(B,A)$-bimodules we have
$$ {_B}(M^\vee)_A \cong {_B}(M^\circledast)_A 
\quad\text{via the correspondence }\quad
t_\ast : f \mapsto t \circ f .$$
Thus we have an isomorphism of left $B$-modules
$$
\Psi : {_B}(M^\circledast \otimes_A X^\circledast) 
          \, {\overset{\approx}{\longrightarrow}}
          \, _B(M^\vee \otimes_A X^\circledast) 
          \, {\overset{{\Phi}}{\longrightarrow}}
          \, _B(X \otimes_A M)^\circledast   
$$
given by
$$
t_\ast(\psi)\otimes_A \theta \mapsto \psi\otimes_A\theta
                              \mapsto \Phi(\psi\otimes_A \theta).
$$
\smallskip
    \item
If finally $A$ and $B$ are block algebras of finite groups,
and $M$ is self-dual, namely,
$M^* \cong M$ as $(A,B)$-bimodules, then as right $B$-modules we have
$$
      (X^* \otimes_A M)_B \ \cong \ [(X \otimes_A M)^*]_B.
$$
\end{enumerate}
\end{Lemma}

\begin{proof}
(i)
Assume first that $B = k$. 
The map $\Phi$ is $k$-linear 
and 
an isomorphism if $M = A$ as a left $A$-module.
Clearly $\Phi$ is compatible with direct sums and
direct summands.
Thus, since $M$ is finitely generated projective as a left
$A$-module, we know that $\Phi$ is an isomorphism of $k$-spaces.
It is easy to see by the definition of $\Phi$ 
that $\Phi$ is a homomorphism of left $B$-modules, too.
A similar argument works works for (ii).

(iii)
It is easy to see that $t_\ast$ is a homomorphism of
$(B,A)$-bimodules, and that $t_\ast$ is injective.
Hence the first assertion follows from
\cite[Proposition 2.7]{Broue2009}.
The second assertion now follows from this together with (i).
Now (iv) follows easily from (iii).
\end{proof}

Finally, a 
fundamental property of the stable equivalences obtained
through \textbf{\ref{gluing}} (see also {\bf\ref{Green-GxG}}) is that
it preserves vertices and sources, and takes indecomposable modules to
their Green correspondents.

\begin{Lemma}\label{SourceVertex}
Let $H$ be a proper subgroup of $G$, and 
let $A$ and $B$ be block algebras of $kG$ and $kH$, respectively.
Now, let $M$ and $M'$ be finitely generated $(A,B)$-
and $(B,A)$-bimodules, respectively, which satisfy the following:
\begin{enumerate}
 \renewcommand{\labelenumi}{\rm{(\arabic{enumi})}}
   \item
${_A}M_B \mid 1_A{\cdot}kG{\cdot}1_B$ and
${_B}{M'}_A \mid 1_B{\cdot}kG{\cdot}1_A$. 
   \item
The pair $(M, M')$ induces a stable equivalence between
${\mathrm{mod}}{\text{-}}A$ and ${\mathrm{mod}}{\text{-}}B$.
\end{enumerate}
Then we get the following:
\begin{enumerate}
  \renewcommand{\labelenumi}{\rm{(\roman{enumi})}}
   \item Assume that $X$ is a non-projective indecomposable $kG$-module
    in $A$ with vertex $Q$. Then there exists a non-projective
    indecomposable $kH$-module $Y$ in $B$, unique up to isomorphism,
    such that $(X \otimes_A M)_B = Y \oplus ({\mathrm{proj}})$, and
    $Q^g$ is a vertex of $Y$ for some element $g \in G$ 
    {\rm{(}}and hence $Q^g \subseteq H${\rm{)}}. 
    Since $Q^g$ is also a vertex of $X$, this means that
    $X$ and $Y$ have at least one vertex in common.
   \item Assume that $Y$ is a non-projective indecomposable $kH$-module
    in $B$ with vertex $Q$. Then there exists a non-projective
    indecomposable $kG$-module $X$ in $A$, unique up to isomorphism,
    such that $(Y \otimes_B {M'})_A = X \oplus ({\mathrm{proj}})$, and
    $Q$ is a vertex of $X$.
   \item Let $X, Y$ and $Q \leqslant H$ be the as in {\rm{(i)}}.
    Then there is an indecomposable $kQ$-module $L$ such that $L$ is a
    source of both $X$ and $Y$. This means that $X$ and $Y$ have 
    at least one source in common.
   \item Let $X, Y$ and $Q \leqslant H$ be the same as in {\rm{(ii)}}.
    Then there is an indecomposable $kQ$-module $L$ such that $L$ is a
    source of both $X$ and $Y$. This means that $X$ and $Y$ have
    at least one source in common.
   \item Let $X, Y$, $Q$ and $L$ be the same as in {\rm{(iii)}}.
    In addition, suppose that $A$ and $B$ have a common defect group $P$
    {\rm{(}}and hence $P \subseteq H${\rm{)}} 
    and that $H \geqslant N_G(P)$. Let $f$
    be the Green correspondence with respect to $(G, P, H)$. If $ Q \in
    \mathfrak A = {\mathfrak A}(G, P, H)$, then we have $(X \otimes_A
    M)_B = f(X) \oplus ({\mathrm{proj}})$.
  \item Let $X$, $Y$, $Q$ and $L$ be the same as in {\rm{(ii)}}.
    Furthermore, as in {\rm{(v)}}, assume that $P$ is a common defect
    group of $A$ and $B$, and that $H \geqslant N_G(P)$, and let $f$
    and $\mathfrak A$ be the same as in {\rm{(v)}}. Now, if $Q \in
    \mathfrak A$, then we have $(Y \otimes_B M')_A = f^{-1}(Y) \oplus
    ({\mathrm{proj}})$.
\end{enumerate} 
\end{Lemma}

\begin{proof}
(i) Clearly, $X \mid {X}{\downarrow_Q\uparrow^G}$. By (2) there exists
a non-projective indecomposable $kH$-module $Y$ in $B$, unique up to
isomorphism, such that $(X \otimes_A M)_B = Y \oplus ({\mathrm{proj}})$.
Hence,
\[ Y \mid X {\otimes_A}M = X {\otimes_{kG}}M \mid
X {\otimes_{kG}}kG_{kH} = {X}{\downarrow_H} \mid
{X}{\downarrow_Q\uparrow^G\downarrow_H} = \bigoplus_{g \in [Q \backslash
G / H]} ({X}{\downarrow_Q})^g{\downarrow}_{Q^g \cap H}{\uparrow^H} .\]
The last equality follows from Mackey Decomposition. Since $Y_{kH}$
is indecomposable, the Krull-Schmidt Theorem yields $Y \mid
({{X}{\downarrow_Q})^g}{\downarrow}_{Q^g \cap H}{\uparrow^H}$ for some
$g \in G$. 
That is, $Y$ is $(Q^g \cap H)$-projective, so that there is
a vertex $R$ of $Y$ such that $R \leqslant Q^g \cap H$. Since $Y \mid
{Y{\downarrow}{_R}}{\uparrow}{^H}$, it holds as above that
\[ X \mid Y{\otimes_B}{M'} = Y {\otimes_{kH}}{M'}
\mid Y{\otimes_{kH}}{kG_{kG}} = {Y{\uparrow}{^G}}
\mid ({Y}{\downarrow}{_R}{\uparrow^H}){\uparrow^G} =
{{Y}{\downarrow}}{_R}{\uparrow}{^G}.\]
Hence, $X$ is $R$-projective, so that there is a vertex $S$ of $X$
with $S \subseteq R$. Since $Q$ is also a vertex of $X$, we have $S =
Q^{g'}$ for some $g' \in G$. Namely, $Q^{g'} \subseteq R$. This implies
that $Q^{g'} = S \subseteq R \subseteq Q^g \cap H \subseteq Q^g$, and
hence $Q^{g'} =R = Q^g \cap H = Q^g$. This yields that $Q^g \subseteq
H$.

(ii) Similar to (i).

(iii) By the assumption, $Q$ is a common vertex of $X$ and $Y$.
Let $L_{kQ}$ be a source of $Y_{kH}$. Then by the proof of (i),
$X \mid Y{\uparrow}{^G} \mid {L}{\uparrow}{^H}{\uparrow}{^G} =
{L}{\uparrow}{^G}$. Hence, $X \mid {{L}{\uparrow}}{^G}$. Since $X$ has
vertex $Q$ and $L$ is an indecomposable $kQ$-module, it follows that $L$
is a source of $X$, too.

(iv) This follows from (iii).

(v) Let $\mathfrak X$, $\mathfrak Y$ and $\mathfrak A$ be 
those with respect to $(G, P, H)$ as in \cite[Chap.4 \S
4]{NagaoTsushima}. Now, let $X$ be an indecomposable $kG$-module in $A$
such that a vertex of $X$ is in $\mathfrak A$. Thus, we can assume that
$Q \in \mathfrak A$. If $X$ is projective then $Q$ is trivial, so that 
the trivial group is not contained in $\mathfrak X$ by the
definition of $\mathfrak A$, a contradiction, since $H \, \not= G$.

Hence, $X$ is non-projective. Thus, we get by (i) and (ii) that there
is a non-projective indecomposable $kH$-module $Y$ in $B$ such that $X
\otimes_A M = Y \oplus ({\mathrm{proj}} \ B{\text{-}}{\mathrm{mod}})$
and that $Y$ also has $Q$ as its vertex. On the other hand, we
know $(X \otimes_A M) \mid X_{kH} = f(X) \oplus ({\mathfrak
Y}{\text{-}}{\mathrm{proj}} \ B{\text{-}}{\mathrm{mod}})$. This
implies that $f(X) \oplus ({\mathfrak Y}{\text{-}}{\mathrm{proj}}
\ B{\text{-}}{\mathrm{mod}}) = Y \oplus ({\mathrm{proj}} \
B{\text{-}}{\mathrm{mod}}) \oplus V$ for a $kH$-module $V$.

Assume that $Y$ is $\mathfrak Y$-projective. Since $Q$ is a vertex of
$Y$, we have $Q \in_H \mathfrak Y$. Hence, we get by \cite[Chap.4
Lemma~4.1(ii)]{NagaoTsushima} that $Q \in \mathfrak X$. Then we have $Q
\not\in \mathfrak A$, a contradiction. Therefore, by the Krull-Schmidt
Theorem, we have $Y \cong f(X)$.

(vi) We get this exactly as in (iii) just by replacing $X$, $M$, and $f$
by $Y$, $M'$, and $f^{-1}$, respectively.
\end{proof}

\end{appendix}

\bigskip

\begin{center}{\bf Acknowledgements}
\end{center}
\small{
The authors thank
the referee for useful comments and remarks on the
first draft of the paper.
The authors are grateful 
to Burkhard K{\"u}lshammer for pointing
out \cite{Willems} to them.
A part of this work was done while the first author was
staying in RWTH Aachen University in 2009 and 2010. 
He is grateful to Gerhard Hiss for his kind hospitality.
For this research the first author was partially
supported by the Japan Society for Promotion of Science (JSPS),
Grant-in-Aid for Scientific Research 
(C)20540008, 2008--2010
; and also
(B)21340003, 2009--2011.
The first author still remembers that in January 1985,
in Bad Honnef Bonn, Germany,
there was a conference and he was informed by Peter Landrock
about the non-principal $2$-block of ${\sf Co}_3$
with defect group $C_2 \times C_2 \times C_2$,
which was interesting because the block is non-principal and
the inertial index is $21$.
}

\end{document}